\documentclass{amsart}
\usepackage{amsfonts,amssymb,latexsym}
\usepackage{amsmath,amsthm}
\usepackage{eucal}
\usepackage[latin1]{inputenc}

\newtheorem{theorem}{Theorem}[section]
\newtheorem{lemma}[theorem]{Lemma}
\newtheorem{proposition}[theorem]{Proposition}

\theoremstyle{definition}

\theoremstyle{remark}
\newtheorem{remark}[theorem]{Remark}
\def\R{{\mathbb R}}
\newcommand{\T}{\mathbb{T}}
\newcommand{\Z}{\mathbb{Z}}

\numberwithin{equation}{section}

\begin{document}

\title[The Cauchy problem for the BO equation in $L^2$]{The Cauchy problem for the Benjamin-Ono equation in $L^2$ revisited}
\author[L. Molinet and D. Pilod]{Luc Molinet $^1$ and Didier Pilod $^2$}


\maketitle

\vspace{-0.5cm}

{\scriptsize \centerline{$^1$ LMPT, Université François Rabelais
Tours, Fédération Denis Poisson-CNRS,}
             \centerline{ Parc Grandmont, 37200 Tours, France.}
             \centerline{email: Luc.Molinet@lmpt.univ-tours.fr}

\vspace{0.1cm}
             \centerline{$^2$ UFRJ, Instituto de Matemática,
             Universidade Federal do Rio de Janeiro,}
              \centerline{Caixa Postal 68530, CEP: 21945-970, Rio
              de Janeiro, RJ, Brazil.}
              \centerline{email: pilod@impa.br, didier@im.ufrj.br}}

\vspace{0.5cm}

\begin{abstract} In a recent work \cite{IK}, Ionescu and Kenig proved that the Cauchy problem associated
to the Benjamin-Ono equation is well-posed in $L^2(\mathbb R)$. In
this paper we give a simpler proof of Ionescu and Kenig's result,
which  moreover provides stronger uniqueness results. In particular,
we prove unconditional well-posedness in $H^s(\mathbb R)$, for
$s>\frac14$.
\end{abstract}

\section{Introduction}

The Benjamin-Ono equation is one of the fundamental equation describing
the evolution of weakly nonlinear internal long waves. It has been derived by Benjamin
\cite{Be} as an approximate model for long-crested unidirectional waves at the interface of
a two-layer system of incompressible inviscid fluids, one being infinitely deep.
In nondimensional variables, the initial value problem (IVP) associated to the Benjamin-Ono
equation (BO) writes as
\begin{equation}\label{BO}
\left\{\begin{array}[pos]{ll}
           \partial_tu+\mathcal{H}\partial^2_xu=u\partial_xu\\
           u(x,0)=u_0(x),
       \end{array} \right.
\end{equation}
where $x \in \mathbb R$ or $\mathbb T$, $t \in \mathbb R$, $u$ is a
real-valued function, and $\mathcal{H}$ is the Hilbert transform,
\textit{i.e.}
\begin{equation} \label{hilbert}
\mathcal{H}f(x)=\text{p.v.}\,\frac1\pi\int_{\mathbb
R}\frac{f(y)}{x-y}dy.
\end{equation}
The Benjamin-Ono equation is, at least formally, completely
integrable \cite{AF} and thus possesses an infinite number of conservation
laws. For example, the momentum and the energy, respectively given
by
\begin{equation} \label{conservationlaws}
M(u)=\int u^2dx, \quad \text{and} \quad
E(u)=\frac12\int\big|D_x^{\frac12}u\big|^2dx+\frac16\int u^3dx,
\end{equation}
are conserved by the flow of \eqref{BO}.

The IVP associated to the Benjamin-Ono equation
presents interesting mathematical difficulties and has been extensively studied in the
recent years. In the continuous case, well-posedness in $H^s(\mathbb R)$ for $s>\frac32$ was
proved by Iorio \cite{Io} by using purely hyperbolic energy
methods (see also \cite{ABFS} for  global well-posedness in the same rage of $ s$ ). Then, Ponce
\cite{Po} derived a local smoothing effect associated to the
dispersive part of the equation, which combined to compactness
methods, enable to reach $s=\frac32$. This technique was refined by Koch
and Tzvetkov \cite{KT} and Kenig and Koenig \cite{KK} who reach respectively $s>\frac54$
and $s>\frac98$. On the other hand Molinet, Saut and Tzvetkov \cite{MST} proved that the
flow map associated to BO, when it exists, fails to be $C^2$ in any Sobolev space
$H^s(\mathbb R)$, $s \in \mathbb R$. This results is based on the fact that the
dispersive smoothing effects of the linear part of BO are not strong enough to control
the low-high frequency interactions appearing in the nonlinearity of \eqref{BO}.
It was improved  by Koch and
Tzvetkov \cite{KK} who showed that the flow map fails even to be uniformly
continuous in $H^s(\mathbb R)$ for  $s>0$ (see  \cite{BL} for
 the same result in the case $ s<-1/2$.)
As the consequence of those results, one cannot solve the Cauchy problem for the
Benjamin-Ono by a Picard iterative method implemented on the integral equation associated to
\eqref{BO} for initial data in the Sobolev space $H^s(\mathbb R)$, $s \in \mathbb R$.
In particular, the methods introduced by Bourgain \cite{Bo2} and
Kenig, Ponce and Vega \cite{KPV2}, \cite{KPV} for the Korteweg- de Vries equation do not apply
directly to the Benjamin-Ono equation.

Therefore, the problem to obtain well-posedness in less regular Sobolev
spaces turns out to be far from trivial. Due to the conservations laws
\eqref{conservationlaws}, $L^2(\mathbb R)$ and $H^{\frac12}(\mathbb R)$ are two natural spaces
where well-posedness is expected. In this direction, a decisive
breakthrough was achieved by Tao \cite{Ta}. By combining a complex variant of
the Cole-Hopf transform (which linearizes Burgers equation) with Strichartz estimates,
he proved well-posedness in $H^1(\mathbb R)$. More precisely, to obtain estimates
at the $H^1$-level, he introduced the new unknown
\begin{equation} \label{BOt3}
w=\partial_xP_{+hi}\big(e^{-\frac i2 F}\big),
\end{equation}
where $F$ is some spatial primitive of $u$ and $P_{+hi}$ denotes the projection on
high positive frequencies. Then $w$ satisfies an equation on the form
\begin{equation} \label{BOt}
\partial_tw-i\partial_x^2w=-\partial_xP_{+hi}\big(\partial_x^{-1}wP_-\partial_xu\big)
+\text{negligible terms}.
\end{equation}
Observe that, thanks to the frequency projections,  the nonlinear term appearing on the right-hand side of \eqref{BOt} does not
 exhibit  any low-high frequency interaction terms. Finally, to inverse this gauge transformation, one gets an equation on
the form
\begin{equation} \label{BOt2}
u=2ie^{\frac i2F}w+\text{negligible terms}.
\end{equation}

Very recently, Burq and Planchon \cite{BP}, and Ionescu and Kenig
\cite{IK} were able to use Tao's ideas in the context of Bourgain's
spaces to prove well-posedness for the Benjamin-Ono equation in
$H^s(\mathbb R)$ for respectively $s> \frac14 $ and $s \ge 0$. The
main difficulty arising here is that Bourgain's spaces do not enjoy
an algebra property, so that one is loosing regularity when
estimating $u$ in terms  of $w$ via equation \eqref{BOt2}. Burq and
Planchon first paralinearized the equation and then used a localized
version of the gauge transformation on the worst nonlinear term. On
the other hand, Ionescu and Kenig decomposed the solution in two
parts: the first one is the smooth solution of BO evolving from the
low frequency part of the initial data, while the second one solves
a dispersive system renormalized by a gauge transformation involving
the first part. The authors then were able to solve the system via a
fixed point argument in a dyadic version of Bourgain's spaces
(already used in the context of wave maps \cite{Tat}) with a special
structure in low frequencies. It is worth noticing that their result
only ensures the uniqueness in the class of limits of smooth
solutions, while Burq and Planchon obtained a stronger uniqueness
result. Indeed, by applying their approach  to the equation
satisfied by the difference of two solutions, they succeed in
proving that  the flow map associated to BO is Lipschitz in a weaker
topology when the initial data belongs to $H^s(\mathbb R)$, $s
>\frac14$.

In the periodic setting, Molinet \cite{Mo}, \cite{Mo2} proved
well-posedness in $H^s(\mathbb T)$ for successively $s \ge \frac12$
and $s \ge 0$. Once again, these works combined Tao's gauge
transformation with estimates in Bourgain's spaces. It should be
pointed out that in the periodic case, one can  assume that $u$ has
mean value zero to define a primitive. Then, it is easy to check by
the mean value theorem that the gauge transformation in \eqref{BOt3}
is Lipschitz from $L^2$ into $L^{\infty}$. This property, which is
not true in the real line, is crucial to prove the uniqueness and
the Lipschitz property of the flow map.

The aim of this paper is to give a simpler  proof of Ionescu and Kenig's
result, which also provides a stronger uniqueness result for the solutions at the $L^2$-
level. It is worth noticing that to reach $ L^2 $ in \cite{IK} or \cite{Mo2} the authors substituted $ u $ in \eqref{BOt3} by the formula given in \eqref{BOt2}.  The good side of this substitution is that now $ u $ will not appear anymore in
 \eqref{BOt3}. On the other hand, it introduces new  technical difficulties to handle the multiplication by
  $ e^{\mp iF/2} $ in Bourgain's spaces. In the present paper we are able to avoid this substitution which will really
    simplify the proof. Our main result is the following
\begin{theorem} \label{theo1}Let $s \ge 0$ be given. \mbox{ } \\
\underline{\it Existence :}
 For all $u_0 \in H^s(\mathbb R)$ and all $ T>0 $, there exists a
 solution
\begin{equation} \label{theo1.1}
u \in C([0,T];H^s(\mathbb R)) \cap X^{s-1,1}_T \cap L^4_TW^{s,4}_x
\end{equation}
of \eqref{BO} such that
\begin{equation} \label{theo1.1b}
w=\partial_xP_{+hi}\big(e^{-\frac{i}{2}F[u]}\big) \in Y_T^s.
\end{equation}
where $ F[u] $ is some primitive of $u$ defined in \eqref{defF}.\vspace{2mm}\\
\underline{\it Uniqueness :} This solution is unique in the following classes :
$$
\begin{array}{lll}
i) & u\in L^\infty(]0,T[;L^2(\R)) \cap L^4(]0,T[\times\R) \mbox{ and } w\in X^{0,\frac12}_T .\\
ii) &  u\in L^\infty(]0,T[;H^{s}(\R)) \cap L^4_T W^{s,4}_x & \mbox{ whenever } s>0 .\\
iii)  &  u\in L^\infty(]0,T[;H^{s}(\R))  & \mbox{ whenever } s>
\frac14 .
\end{array}
$$
Moreover, $u\in C_b(\mathbb R;L^2(\mathbb R))$ and the flow map
data-solution $:u_0 \mapsto u$ is continuous from $H^s(\mathbb R)$
into $C([0,T];H^s(\mathbb R))$.
\end{theorem}
Note that  above $H^s(\mathbb R)$ denotes the space of all real-valued
functions with the usual norm, $X_T^{s,b}$ and $Y^s_T$
are  Bourgain spaces  defined
in Subsection \ref{FS}, while the primitive $F[u]$ of $u$
is defined in Subsection \ref{GT}.
\begin{remark}  Since  the function spaces in the uniqueness class i) are reflexive and since $ \partial_x
P_{+hi}\big(e^{-\frac{i}{2}F[u_n]}\big) $ converges to $\partial_x  P_{+hi}\big(e^{-\frac{i}{2}F[u]}\big) $ in $ L^\infty(]-T,T[;L^2(\R)) $ whenever
 $ u_n $ converges to $ u$ in $ L^\infty(]-T,T[;L^2(\R)) $, our result clearly implies the uniqueness in the class of $ L^\infty(]-T,T[;L^2(\R)) $-limits of smooth solutions.
\end{remark}
\begin{remark}
 It is worth noticing that for $ s>0 $ we get a uniqueness class without condition on $ w$ (see \cite{BP}  for the case $ s>\frac14$).
\end{remark}
\begin{remark}
 According to $ iii) $ we get unconditional well-posedness in $ H^s(\R) $ for  $ s>\frac14 $. This implies  in particular the uniqueness of  the (energy) weak solutions  that belong to $ L^\infty(\R; H^{1/2}(\R)) $. These solutions  are constructed by regularizing the equation and passing to the limit as the regularizing coefficient goes to $0$ (taking into account  some energy estimate for the regularizing equation related to the energy conservation of \eqref{BO}) . \end{remark}

Our proof also combines Tao's ideas with the use of Bourgain's spaces. Actually,
it follows closely the strategy introduced by the first author in \cite{Mo}.
The main new  ingredient is a bilinear estimate for the nonlinear term appearing in \eqref{BOt},
which allows to recover one derivative at the $L^2$-level. It is interesting to note that, at the
$H^s$-level with $s>0$, this estimate follows from the Cauchy-Schwarz method introduced by Kenig, Ponce and Vega in \cite{KPV} (see the appendix for the use of this method in some region of integration). To reach $L^2$,  one of the main difficulty is that we cannot substitute  the Fourier transform of $u $  by its modulus in the bilinear estimate since
 we  are not able to prove that  $\mathcal{F}^{-1}(|\hat{u}|)$ belongs $L^4_{x,t}$ but only that  $u$ belongs to $ L^4_{x,t} $ .  To  overcome this difficulty we use a Littlewood-Paley decomposition of the functions and carefully divide the domain of integration into suitable  disjoint subdomains.

To obtain our uniqueness result, following the same method as in the periodic setting, we derive a Lipschitz bound
  for  the gauge transformation from some affine subspaces of
 $L^2(\R) $ into  $L^{\infty}(\mathbb R)$ . Recall that  this is clearly not possible for general initial data since it would imply  the uniform continuity of the flow-map. The main idea is to notice that such   Lipschitz bound holds for solutions emanating from  initial data having the same low frequency part and this is sufficient for our purpose.

Let us point out some applications. First our uniqueness result allows to really simplify the  proof of the continuity  of the flow map associated to the Benjamin-Ono equation for the weak topology of $  L^2(\mathbb R)$. This result was recently proved by Cui and Kenig \cite{CK}.

It is also interesting to observe that the method of proof used here still works in the periodic setting, and thus, we reobtain the well-posedness result \cite{Mo2}  in a simpler way. Moreover, as in the continuous case, we also prove new uniqueness results (see Theorem \ref{theo1per} below). In particular, we get unconditional well-posedness in $H^s(\T)$ as soon as $s \ge \frac12$.

Finally, we believe that this technique may be useful for another nonlinear dispersive equations presenting the same kind of difficulties as the Benjamin-Ono equation. For example, consider the higher-order Benjamin-Ono equation
\begin{equation} \label{hoBO}
\partial_tv-b\mathcal{H}\partial^2_xv+a \partial_x^3v=cv\partial_xv-d
           \partial_x(v\mathcal{H}\partial_xv+\mathcal{H}(v\partial_xv)),
\end{equation}
where $x$, $t \in \mathbb R$, $v$ is a real-valued function,
$a \in \mathbb R$, $b, \ c$ and $d$ are positive constants.
The equation above corresponds to a second order approximation model of the same
phenomena described by the Benjamin-Ono equation.
It was derived by Craig, Guyenne and Kalisch \cite{CGK} using a Hamiltonian perturbation theory, and possesses an energy at the $H^1$-level. As for the Benjamin-Ono equation, the flow map associated to \eqref{hoBO} fails to be smooth in any Sobolev space $H^s(\mathbb R)$, $s \in \mathbb R$ \cite{Pi}.
Recently, the Cauchy problem associated to \eqref{hoBO} was proved to be well-posed in $H^2(\mathbb R)$ \cite{LPP}. In a forthcoming paper, the authors will show that it is actually well-posed in the energy space $H^1(\mathbb R)$.

This paper is organized as follows: in the next section, we introduce the notations, define the function spaces and recall some classical linear estimates. Section 3 is devoted to the key nonlinear estimates, which are used in Section 4 to prove the main part of Theorem \ref{theo1}, while the assertions  $i)$ and $ii)$
 are proved  in Section 5. In Section 6, we give a simple  proof of the continuity of the flow-map for the weak
  $ L^2(\R) $-topology whereas Section 7 is devoted to some comments and new results in the periodic case. Finally,
 in the appendix we prove the bilinear estimate used in Section 5.

\section{Notation, function spaces and preliminary estimates}

\subsection{Notation}
For any positive numbers $a$ and $b$, the notation $a \lesssim b$ means that there exists a positive constant
$c$ such that $a \le c b$. We also denote $a \sim b$ when $a \lesssim b$ and $b \lesssim a$. Moreover, if $\alpha \in \mathbb R$, $\alpha_+$, respectively $\alpha_-$, will denote a number slightly greater, respectively lesser, than $\alpha$.

For $u=u(x,t) \in \mathcal{S}(\mathbb R^2)$,
$\mathcal{F}u=\widehat{u}$ will denote its space-time Fourier
transform, whereas $\mathcal{F}_xu=(u)^{\wedge_x}$, respectively
$\mathcal{F}_tu=(u)^{\wedge_t}$, will denote its Fourier transform
in space, respectively in time. For $s \in \mathbb R$, we define the
Bessel and Riesz potentials of order $-s$, $J^s_x$ and $D_x^s$, by
\begin{displaymath}
J^s_xu=\mathcal{F}^{-1}_x\big((1+|\xi|^2)^{\frac{s}{2}}
\mathcal{F}_xu\big) \quad \text{and} \quad
D^s_xu=\mathcal{F}^{-1}_x\big(|\xi|^s \mathcal{F}_xu\big).
\end{displaymath}

Throughout the paper, we fix a cutoff function $\eta$ such that
\begin{displaymath}
\eta \in C_0^{\infty}(\mathbb R), \quad 0 \le \eta \le 1, \quad
\eta_{|_{[-1,1]}}=1 \quad \mbox{and} \quad  \mbox{supp}(\eta)
\subset [-2,2].
\end{displaymath}
We define
\begin{displaymath}
\phi(\xi):=\eta(\xi)-\eta(2\xi) \quad \mbox{and} \quad
\phi_{2^l}(\xi):=\phi(2^{-l}\xi).
\end{displaymath}
Any summations over capitalized variables such as $N$ are presumed to be dyadic with $N \ge 1$, \textit{i.e.},
these variables range over numbers of the form $2^n$, $n \in \mathbb Z_{+}$. Then, we have that
\begin{displaymath}
\sum_{N}\phi_N(\xi)=1-\eta(2\xi), \ \forall \xi \neq 0 \quad \mbox{and}
\quad \mbox{supp} \, (\phi_N) \subset \{\frac{N}{2}\le |\xi| \le 2N\}.
\end{displaymath}
Let us define the Littlewood-Paley multipliers by
\begin{displaymath}
P_Nu=\mathcal{F}^{-1}_x\big(\phi_N\mathcal{F}_xu\big),
\quad \text{and} \quad P_{\ge N}:=\sum_{K \ge N} P_{K}.
\end{displaymath}
Moreover, we also define the operators $P_{hi}$, $P_{HI}$, $P_{lo}$
and $P_{LO}$ by
\begin{displaymath}
P_{hi}=\sum_{N\ge 2 } P_N, \quad  P_{HI}=\sum_{N \ge 8}P_N, \quad
P_{lo}= 1-P_{hi}, \quad \text{and} \quad P_{LO}= 1-P_{HI}.
\end{displaymath}

Let $P_+$ and $P_-$ the projection on respectively the positive and
the negative Fourier frequencies. Then
\begin{displaymath}
P_{\pm}u=\mathcal{F}^{-1}_x\big(\chi_{\mathbb
R_{\pm}}\mathcal{F}_xu\big),
\end{displaymath}
and we also denote $P_{\pm hi}=P_{\pm}P_{hi}$, $P_{\pm
HI}=P_{\pm}P_{HI}$, $P_{\pm lo}=P_{\pm}P_{lo}$ and $P_{\pm
LO}=P_{\pm}P_{LO}$. Observe that $P_{hi}$, $P_{HI}$, $P_{lo}$ and
$P_{LO}$ are bounded operators on $L^p(\mathbb R)$ for $1\le
p\le\infty$, while $P_{\pm}$ are only bounded on $L^p(\mathbb R)$
for $1 < p < \infty$. We also note that
\begin{displaymath}
\mathcal{H}=-iP_++iP_-.
\end{displaymath}

Finally, we denote by $U(\cdot)$ the free group associated with the linearized Benjamin-Ono equation, which is to say,
\begin{displaymath}
\mathcal{F}_x\big(U(t)f \big)(\xi)=e^{-it|\xi|\xi}\mathcal{F}_xf(\xi).
\end{displaymath}

\subsection{Function spaces} \label{FS}
For $1 \le p \le \infty$, $L^p(\mathbb R)$ is the usual Lebesgue space with the norm $\|\cdot\|_{L^p}$, and
for $s \in \mathbb R$ , the real-valued Sobolev spaces $H^s(\mathbb R)$ and $W^{s,p}(\mathbb R)$ denote the spaces of all real-valued functions with the usual norms
\begin{displaymath}
\|f\|_{H^s}=\|J^su\|_{L^2}
\quad \text{and} \quad
\|f\|_{W^{s,p}}=\|J^s_x f\|_{L^p}.
\end{displaymath}
For $1 < p< \infty$,  we define the
space $\tilde{L}^p$
\begin{displaymath}
\|f\|_{\tilde{L}^p}=\|P_{lo}f\|_{L^p}+\Big(\sum_{N}\|P_{N}f\|_{L^p}^2\Big)^{\frac{1}{2}}.
\end{displaymath}
Observe that when $p \ge 2$, the Littlewood-Paley theorem on the square function and Minkowski's inequality imply that the injection $\tilde{L}^p \hookrightarrow L^p$ is continuous.
Moreover, if $u=u(x,t)$ is a real-valued function defined for $x \in \mathbb R$ and $t$ in the time interval $[0,T]$, with $T>0$, if $B$ is one of the spaces defined above and $1 \le p \le \infty$, we will define the mixed space-time spaces $L^p_TB_x$, respectively $L^p_tB_x$, by the norms
\begin{displaymath}
\|u\|_{L^p_TB_x} =\Big( \int_0^T\|u(\cdot,t)\|_{B}^pdt\Big)^{\frac1p}
\quad \text{respectively} \quad \|u\|_{L^p_tB_x} =\Big( \int_{\mathbb R}\|u(\cdot,t)\|_{B}^pdt\Big)^{\frac1p}.
\end{displaymath}

For $s$, $b \in \mathbb R$, we introduce the Bourgain spaces $X^{s,b}$ and $Z^{s,b}$ related to the Benjamin-Ono equation as the completion of the Schwartz space $\mathcal{S}(\mathbb R^2)$ under the norms
\begin{equation} \label{Bourgain}
\|u\|_{X^{s,b}} := \left(
\int_{\mathbb{R}^2}\langle\tau+|\xi|\xi\rangle^{2b}\langle
\xi\rangle^{2s}|\widehat{u}(\xi, \tau)|^2 d\xi d\tau \right)^{1/2},
\end{equation}
\begin{equation}\label{Bourgain2}
\|u\|_{Z^{s,b}} := \left(
\int_{\mathbb{R}}\Big(\int_{\mathbb R}\langle\tau+|\xi|\xi\rangle^{b}\langle
\xi\rangle^{s}|\widehat{u}(\xi, \tau)|d\tau \Big)^2 d\xi  \right)^{1/2},
\end{equation}
\begin{equation} \label{Bourgain4}
\|u\|_{\widetilde{Z}^{s,b}}=\|P_{lo}u\|_{Z^{s,b}}+\left(\sum_{N}\|P_Nu\|_{Z^{s,b}}^2 \right)^{\frac12},
\end{equation}
and
\begin{equation} \label{Bourgain3}
\|u\|_{Y^{s}}=\|u\|_{X^{s,\frac12}}+\|u\|_{\widetilde{Z}^{s,0}},
\end{equation}
where $\langle x\rangle:=1+|x|$. We will also use the localized (in time) version of these spaces. Let $T>0$ be a positive time and $\|\cdot\|_{B}=\|\cdot\|_{X^{s,b}}$, $\|\cdot\|_{\widetilde{Z}^{s,b}}$ or $\|\cdot\|_{Y^{s}}$. If $u: \mathbb R \times [0,T]\rightarrow \mathbb C$, then
\begin{displaymath}
\|u\|_{B_T}:=\inf \{\|\tilde{u}\|_{B} \ | \ \tilde{u}:
\mathbb R \times \mathbb R \rightarrow \mathbb C, \tilde{u}|_{\mathbb R \times [0,T]} = u\}.
\end{displaymath}
it is worth recalling that
\begin{displaymath}
Y^s_T \hookrightarrow Z^{s,0}_T \hookrightarrow C([0,T];H^s(\mathbb R)).
\end{displaymath}

\subsection{Linear estimates}
In this subsection, we recall some linear estimates in Bourgain's
spaces which will be needed later. The first ones are well-known
(cf. \cite{GTV} for example).
\begin{lemma}[Homogeneous linear estimate] \label{prop1.1}
Let $s \in \mathbb R$. Then
\begin{equation} \label{prop1.1.2}
\|\eta(t)U(t)f\|_{Y^s} \lesssim\|f\|_{H^s}.
\end{equation}
\end{lemma}
\begin{lemma}[Non-homogeneous linear estimate] \label{prop1.2}
Let $s \in \mathbb R$. Then for any $0<\delta<1/2$,
\begin{equation} \label{prop1.2.1}
\big\|\eta(t)\int_0^tU(t-t')g(t')dt'\big\|_{X^{s,\frac12+\delta}} \lesssim  \|g\|_{X^{s,-\frac12+\delta}}
\end{equation}
and
\begin{equation} \label{prop1.2.2}
\big\|\eta(t)\int_0^tU(t-t')g(t')dt'\big\|_{Y^s} \lesssim
\|g\|_{X^{s,-\frac12}}
+\|g\|_{\widetilde{Z}^{s,-1}}.
\end{equation}
\end{lemma}

\begin{proof}[Proof of Lemmas \ref{prop1.1} and \ref{prop1.2}] The proof of
Lemmas \ref{prop1.1} and \ref{prop1.2} is a direct consequence of the classical
linear estimates for $X^{s,b}$ and $Z^{s,b}$ and the fact that
\begin{displaymath}
\|u\|_{X^{s,b}}=\|P_{lo}u\|_{X^{s,b}}+\big(\sum_{N} \|P_N
u\|_{X^{s,b}}^2\big)^{1/2}.
\end{displaymath}
\end{proof}

\begin{lemma} \label{prop1.3b}
For any $T>0$, $s \in \mathbb R$ and for all $-\frac12< b' \le b
<\frac12$, it holds
\begin{equation} \label{prop1.3b.1}
\|u\|_{X^{s,b'}_T} \lesssim T^{b-b'}\|u\|_{X^{s,b}_T}.
\end{equation}
\end{lemma}

The following Bourgain-Strichartz estimates will also be useful.
\begin{lemma} \label{bourgstrich}
It holds that
\begin{equation} \label{bourgstrich.2}
\|u\|_{L^4_{x,t}} \lesssim \|u\|_{\widetilde{L}^4_{x,t}} \lesssim
\|u\|_{X^{0,\frac38}}
\end{equation}
and for any $T>0$ and $\frac38 \le b \le \frac12$,
\begin{equation} \label{bourgstrich.3}
\|u\|_{L^4_{x,T}} \lesssim T^{b-\frac38}\|u\|_{X^{0,b}_T}.
\end{equation}
\end{lemma}
\begin{proof}

Estimate \eqref{bourgstrich.2} follows directly by applying the estimate
\begin{displaymath}
\|u\|_{L^4_{x,t}} \lesssim \|u\|_{X^{0,\frac38}},
\end{displaymath}
proved in the appendix of \cite{Mo},
to each dyadic block on the left-hand side of \eqref{bourgstrich.2}.

To prove \eqref{bourgstrich.3}, we choose an extension $\tilde{u}
\in X^{0,b}$ of $u$ such that $\|\tilde{u}\|_{X^{0,b}} \le
2\|u\|_{X^{0,b}_T}$. Therefore, it follows from \eqref{prop1.3b.1}
and \eqref{bourgstrich.2} that
\begin{displaymath}
\|u\|_{L^4_{x,T}} \le \|\tilde{u}\|_{L^4_{x,t}} \lesssim \|\tilde{u}\|_{X^{0,\frac38}} \lesssim T^{b-\frac38}\|u\|_{X_T^{0,b}}.
\end{displaymath}
\end{proof}

\subsection{Fractional Leibniz's rules}
First we state the classical fractional Leibniz rule estimate
derived by Kenig, Ponce and Vega (See Theorems A.8 and A.12 in \cite{KPV2}).
\begin{proposition} \label{leibrule}
Let $0<\alpha<1$, $p, \ p_1, \ p_2 \in (1,+\infty)$ with $\frac1{p_1}+\frac1{p_2}=\frac1p$
and $\alpha_1, \ \alpha_2 \in [0,\alpha]$ with $\alpha=\alpha_1+\alpha_2$. Then,
\begin{equation} \label{leibrule1}
\big\|D^{\alpha}_x(fg)-fD^{\alpha}_xg-gD^{\alpha}_xf \big\|_{L^p}
\lesssim \|D_x^{\alpha_1}g\|_{L^{p_1}}\|D^{\alpha_2}_xf\|_{L^{p_2}}.
\end{equation}
Moreover, for $\alpha_1=0$, the value $p_1=+\infty$ is allowed.
\end{proposition}

The next estimate is a frequency localized version of estimate \eqref{leibrule1}
in the same spirit as Lemma 3.2 in \cite{Ta}. It allows to share most of the fractional
derivative in the first term on the right-hand side of \eqref{lemma2.1}.
\begin{lemma} \label{lemma2}
Let $\alpha \ge 0$ and $1<q<\infty$. Then,
\begin{equation} \label{lemma2.1}
\big\|D_x^{\alpha}P_+\big(fP_-\partial_xg\big) \big\|_{L^q} \lesssim
\|D_x^{{\alpha}_1}f\|_{L^{q_1}}\|D_x^{{\alpha}_2}g\|_{L^{q_2}},
\end{equation}
with $1<q_i<\infty$, $\frac1{q_1}+\frac1{q_2}=\frac1q$ and $\alpha_1
\ge \alpha$, $\alpha_2 \ge 0$ and $\alpha_1+\alpha_2=1+\alpha$.
\end{lemma}

\begin{proof} See Lemma 3.2 in \cite{Mo}.
\end{proof}

Finally, we derive an estimate to handle the multiplication by a
term on the form $e^{\pm\frac{i}{2}F}$, where $F$ is a real-valued function,
in fractional Sobolev spaces.
\begin{lemma} \label{lemma1}
Let $2 \le q <\infty$ and $0 \le \alpha \le \frac1q$. Consider $F_1$ and
$F_2$ two real-valued functions such that $u_j=\partial_xF_j$ belongs
to $L^2(\mathbb R)$ for $j=1, \,2$. Then, it holds that
\begin{equation} \label{lemma1.1}
\|J^{\alpha}_x\big(e^{\pm \frac{i}{2}F_1} g\big)\|_{L^q} \lesssim
(1+\|u_1\|_{L^2})\|J^{\alpha}_xg\|_{L^q},
\end{equation}
and
\begin{equation} \label{lemma1.1b}
\begin{split}
\|J^{\alpha}_x&\big(\big(e^{\pm \frac{i}{2}F_1}-e^{\pm \frac{i}{2}F_2}\big)
g\big)\|_{L^q} \\ &\lesssim \Big(\|u_1-u_2\|_{L^2}+\|e^{\pm
\frac{i}{2}F_1}-e^{\pm
\frac{i}{2}F_2}\|_{L^{\infty}}(1+\|u_1\|_{L^2})\Big)\|J^{\alpha}_xg\|_{L^q}.
\end{split}
\end{equation}
\end{lemma}

\begin{proof} In the case $\alpha=0$, we deduce from H\"older's inequality that
\begin{equation} \label{lemma1.2}
\|e^{\pm \frac{i}{2}F_1} g\|_{L^q} \le
\|g\|_{L^q},
\end{equation}
since $F_1$ is real-valued. Therefore we can assume that $0<\alpha\le
\frac1q$ and it is enough to bound $\|D^{\alpha}_x\big(e^{\pm
\frac{i}{2}F_1} g\big)\|_{L^q}$. First, we observe that
\begin{equation} \label{lemma1.3}
\|D^{\alpha}_x\big(e^{\pm \frac{i}{2}F_1} g\big)\|_{L^q} \le
\|D^{\alpha}_x\big(P_{lo}e^{\pm \frac{i}{2}F_1} g\big)\|_{L^q}
+\|D^{\alpha}_x\big(P_{hi}e^{\pm \frac{i}{2}F_1} g\big)\|_{L^q}
\end{equation}
Estimate \eqref{leibrule1} and Bernstein's inequality imply that
\begin{equation} \label{lemma1.4}
\begin{split}
\|D^{\alpha}_x\big(P_{lo}&e^{\pm \frac{i}{2}F_1} g\big)\|_{L^q}
\\&\lesssim \|P_{lo}e^{\pm
\frac{i}{2}F_1}\|_{L^{\infty}}\|D_x^{\alpha}g\|_{L^q}
+ \|D^{\alpha}_xP_{lo}e^{\pm
\frac{i}{2}F_1}\|_{L^{\infty}}\|g\|_{L^q}
\lesssim \|J_x^{\alpha}g\|_{L^q}.
\end{split}
\end{equation}
On the other hand, by using again estimate \eqref{leibrule1}, we get that
\begin{displaymath}
\|D^{\alpha}_x\big(P_{hi}e^{\pm \frac{i}{2}F_1} g\big)\|_{L^q} \lesssim \|P_{hi}e^{\pm
\frac{i}{2}F_1}\|_{L^{\infty}}\|D^{\alpha}_xg\|_{L^q}+\|g\|_{L^{q_1}}
\|D^{\alpha}_xP_{hi}e^{\pm
\frac{i}{2}F_1}\|_{L^{q_2}},
\end{displaymath}
with $\frac{1}{q_1}=\frac1q-\alpha$, $\frac1{q_2}=\alpha$, so that
$\frac{1}{q_1}+\frac{1}{q_2}=\frac1q$. Then, it follows from the
facts that $F_1$ is real-valued, $\partial_xF_1=u_1$ and the Sobolev
embedding that
\begin{equation} \label{lemma1.5}
\begin{split}
\|D^{\alpha}_x\big(P_{hi}e^{\pm \frac{i}{2}F_1} g\big)\|_{L^q} &\lesssim
\|D^{\alpha}_xg\|_{L^q}+\|J_x^{\alpha}g\|_{L^{q}}\|D^{\alpha+\frac12}_xP_{hi}e^{\pm
\frac{i}{2}F_1}\|_{L^{2}} \\
& \lesssim \|J^{\alpha}_xg\|_{L^q}\big(1+\|\partial_xe^{\pm
\frac{i}{2}F_1}\|_{L^2} \big) \\ &
\lesssim
\|J^{\alpha}_xg\|_{L^q}\big(1+\|u_1\|_{L^2} \big).
\end{split}
\end{equation}
The proof of estimate \eqref{lemma1.1} is concluded gathering \eqref{lemma1.2}--\eqref{lemma1.5}.

Estimate \eqref{lemma1.1b} can be obtained exactly in the same way, using that
\begin{equation} \label{lemma1.6}
\|\partial_x\big(e^{\pm \frac{i}{2}F_1}-e^{\pm \frac{i}{2}F_2}\big)\|_{L^2}
\lesssim \|u_1-u_2\|_{L^2}+\|e^{\pm \frac{i}{2}F_1}-e^{\pm \frac{i}{2}F_2}\|_{L^{\infty}}\|u_1\|_{L^2}.
\end{equation}
\end{proof}

\section{A priori estimates in $H^s(\mathbb R)$ for $s\ge0$}

In this section we will derive \textit{a priori} estimates on a solution $u$ to
\eqref{BO} at the $H^s$- level, for $s \ge 0$. First, following Tao in \cite{Ta},
we perform a nonlinear transformation on the equation to weaken the high-low frequency
interaction in the nonlinearity.  Furthermore, since we want to reach $L^2$, we will need
to use Bourgain spaces. This requires a new bilinear estimate which will be derive in
Subsection 3.2.

\subsection{The gauge transformation} \label{GT}
Let $u$ be a solution to the equation in \eqref{BO}. First, we
construct a spatial primitive  $F=F[u]$ of $u$, \textit{i.e.}
$\partial_xF=u$, that satisfies the equation :
\begin{equation} \label{gauge0}
\partial_tF=-\mathcal{H}\partial^2_xF+\frac12(\partial_xF)^2.
\end{equation}
It is worth noticing that these two properties defined $ F $ up to a constant. In order to construct $F$ for $u$ with low regularity,
we use the construction of Burq and Planchon in \cite{BP}.
Consider $\psi \in C_0^{\infty}(\mathbb R)$ such that $\int_{\mathbb
R}\psi(y)dy=1$ and define
\begin{equation}\label{defF}
F(x,t)=\int_{\mathbb R}\psi(y)\Big(\int_y^xu(z,t)dz \Big)dy+G(t),
\end{equation}
as a mean of antiderivatives of $u$. Obviously, $\partial_xF=u$
and
\begin{displaymath}
\begin{split}
\partial_tF(x,t)&=\int_{\mathbb R}\psi(y)\Big(\int_y^x\partial_tu(z,t)dz \Big)dy+G'(t)\\
&=\int_{\mathbb
R}\psi(y)\Big(\int_y^x\big(-\mathcal{H}\partial_z^2u(z,t)
+\frac12\partial_z(u(z,t)^2) \big)dz \Big)dy+G'(t)\\
&=-\mathcal{H}\partial_xu(x,t)+\frac12u(x,t)^2
+\int_{\mathbb R}\big(\mathcal{H}\psi'(y)u(y,t)-\psi(y)\frac12u(y,t)^2\big)dy+G'(t).
\end{split}
\end{displaymath}
Therefore we choose $G$ as
\begin{displaymath}
G(t)=\int_0^t\int_{\mathbb
R}\big(-\mathcal{H}\psi'(y)u(y,s)+\psi(y)\frac12u(y,s)^2\big)dyds,
\end{displaymath}
to ensure that \eqref{gauge0} is satisfied.
Observe that this construction makes sense for $u \in L^2_{\text{loc}}(\mathbb R^2)$.
Next, we introduce the new unknown
\begin{equation} \label{gauge1}
W=P_{+hi}\big(e^{-\frac{i}{2}F}\big) \quad \text{and} \quad
w=\partial_xW=-\frac{i}{2}P_{+hi}\big(e^{-\frac{i}{2}F}u\big).
\end{equation}
Then, it follows from \eqref{gauge0} and the identity $\mathcal{H}=-i(P_+-P_-)$ that
\begin{displaymath}
\begin{split}
\partial_tW+\mathcal{H}\partial_x^2W&=\partial_tW-i\partial_x^2W
=-\frac{i}2P_{+hi}\big(e^{-\frac{i}{2}F}(\partial_tF-i\partial_x^2F-\frac12(\partial_xF)^2)\big) \\
&=-P_{+hi}\big(WP_-\partial_xu\big)-P_{+hi}\big(P_{lo}e^{-\frac{i}{2}F}P_-\partial_xu\big),
\end{split}
\end{displaymath}
since the term
$-P_{+hi}\big(P_{-hi}e^{-\frac{i}{2}F}P_-\partial_xu\big)$ cancels
due to the frequency localization. Thus, it follows differentiating
that
\begin{equation} \label{gauge2}
\partial_tw-i\partial_x^2w=-\partial_xP_{+hi}\big(WP_-\partial_xu\big)
-\partial_xP_{+hi}\big(P_{lo}e^{-\frac{i}{2}F}P_-\partial_xu\big).
\end{equation}
On the other hand, one can write $u$ as
\begin{equation} \label{gauge3}
\begin{split}
u&=F_x=e^{\frac{i}{2}F}e^{-\frac{i}{2}F}F_x=2ie^{\frac{i}{2}F}\partial_x\big(e^{-\frac{i}{2}F}\big)
\\ &=2ie^{\frac{i}{2}F}w-e^{\frac{i}{2}F}P_{lo}(e^{-\frac{i}{2}F}u)
-e^{\frac{i}{2}F}P_{-hi}(e^{-\frac{i}{2}F}u),
\end{split}
\end{equation}
so that it follows from the frequency localization
\begin{equation} \label{gauge3}
\begin{split}
P_{+HI}u&=2iP_{+HI}\big(e^{\frac{i}{2}F}w\big)-P_{+HI}\big(P_{+hi}
e^{\frac{i}{2}F}P_{lo}(e^{-\frac{i}{2}F}u)\big)
\\ & \quad
+2iP_{+HI}\big(P_{+HI}e^{\frac{i}{2}F}\partial_xP_{-hi}e^{-\frac{i}{2}F}\big).
\end{split}
\end{equation}

\begin{remark}
Note that the use of $P_{+HI}$ allows to replace $e^{\frac{i}{2}F}$
by $P_{+hi}e^{\frac{i}{2}F}$ in the second term on the right-hand
side of \eqref{gauge3}. This fact will be useful to obtain at least
a quadratic term in $\|u\|_{L^{\infty}_TL^2_x}$ on the right-hand
side of estimate \eqref{apriori u.2} in Proposition \ref{apriori u}.
\end{remark}

Then, we have the following \textit{a priori} estimates on $u$ in terms of $w$.
\begin{proposition} \label{apriori u}
Let $0 \le s \le 1$, $0<T\le1$, $0\le\theta\le 1$ and $u$ be a
solution to \eqref{BO} in the time interval $[0,T]$. Then, it holds
that
\begin{equation} \label{apriori u.1}
\|u\|_{X^{s-\theta,\theta}_T} \lesssim \|u\|_{L
^{\infty}_TH^s_x}+\|u\|_{L^4_{T,x}}\|J^s_xu\|_{L^4_{T,x}}.
\end{equation}
Moreover, if $0 \le s \le \frac14$, it holds that
\begin{equation} \label{apriori u.2}
\|J^s_xu\|_{L^p_TL^q_x} \lesssim \|u_0\|_{L^2}
+\big(1+\|u\|_{L^{\infty}_TL^2_x}\big)\big(\|w\|_{Y^{s}_T}+\|u\|_{L^{\infty}_TL^2_x}^2\big),
\end{equation}
for $(p,q)=(\infty,2)$ or $(4,4)$.
\end{proposition}

\begin{remark} It is worth notice that \eqref{apriori u.2} could be
rewritten in a convenient form for $ s \ge \frac14$ (cf. \cite{Mo}).
\end{remark}

\begin{proof}
We begin with the proof of estimate \eqref{apriori u.1}
and construct a suitable extension in time $\tilde{u}$ of $u$.
First, we consider $v(t)=U(-t)u(t)$ on the time interval $[0,T]$ and extend
$v$ on $[-2,2]$ by setting $\partial_tv=0$ on $[-2,2] \setminus [0,T]$.
Then, it is pretty clear that
\begin{displaymath}
\|\partial_tv\|_{L^2_{[-2,2]}H_x^r}=\|\partial_tv\|_{L^2_{T}H_x^r},
\quad \text{and} \quad \|v\|_{L^2_{[-2,2]}H_x^r} \lesssim
\|v\|_{L^{\infty}_{T}H_x^r},
\end{displaymath}
for all $r \in \mathbb R$. Now, we define
$\tilde{u}(x,t)=\eta(t)U(t)v(t)$. Obviously, it holds
\begin{equation} \label{apriori u.1b}
\|\tilde{u}\|_{X^{s-1,1}}\lesssim
\|\partial_tv\|_{L^2_{[-2,2]}H_x^{s-1}}
+\|v\|_{L^2_{[-2,2]}H_x^{s-1}}
\lesssim \|\partial_tv\|_{L^2_{T}H_x^{s-1}}
+\|v\|_{L^{\infty}_{T}H_x^{s-1}},
\end{equation}
and
\begin{equation} \label{apriori u.1bb}
\|\tilde{u}\|_{X^{s,0}} \lesssim \|v\|_{L^2_{[-2,2]}H^s_x}\lesssim
\|v\|_{L^{\infty}_TH^s_x}=\|u\|_{L^{\infty}_{T}H_x^s}.
\end{equation}
Then, it is deduced interpolating between \eqref{apriori u.1b} and \eqref{apriori u.1bb} and using the identity
\begin{displaymath}
\partial_tv=
\mathcal{H}\partial_x^2 U(-t)u+U(-t)\partial_t u= U(-t) \Bigl[
\mathcal{H}\partial_x^2 u+\partial_tu\Bigr],
\end{displaymath}
that
\begin{equation} \label{apriori u.1bbb}
\|\tilde{u}\|_{X^{s-\theta,\theta}}\lesssim
\|\partial_tu+\mathcal{H}\partial_x^2u\|_{L^2_TH^{s-1}_x}
+\|u\|_{L^{\infty}_TH^s_x},
\end{equation}
for all $0 \le \theta \le 1$. Therefore, the fact that $u$ is a solution to \eqref{BO} and the fractional Leibniz rule (cf. \cite{KPV2}) yield
\begin{displaymath}
\|\tilde{u}\|_{X^{s-\theta,\theta}}\lesssim
\|u\|_{L^{\infty}_TH^s_x}+\|u\|_{L^4_{x,T}}\|J^s_xu\|_{L^4_{x,T}},
\end{displaymath}
which concludes the proof of \eqref{apriori u.1} since $\tilde{u}$ extends $u$ outside of $[0,T]$.

Next, we turn to the proof of \eqref{apriori u.2}. Let $0 \le T \le
1$, $0 \le s \le \frac14$, $(p,q)=(\infty,2)$ or $(4,4)$ and $u$ a
smooth solution to the equation in \eqref{BO}. Since $u$ is
real-valued, it holds $P_-u=\overline{P_+u}$, so that
\begin{equation} \label{apriori u.3}
\|J_x^su\|_{L^p_TL^q_x} \lesssim
\|P_{LO}u\|_{L^p_TL^q_x}+\|D_x^sP_{+HI}u\|_{L^p_TL^q_x}.
\end{equation}

To estimate the second term on the right-hand side of \eqref{apriori u.3}, we use \eqref{gauge3} to deduce that
\begin{displaymath}
\begin{split}
\|D_x^sP_{+HI}u\|_{L^p_TL^q_x}&\lesssim
\big\|D_x^sP_{+HI}\big(e^{\frac{i}{2}F}w\big)\big\|_{L^p_TL^q_x}
+\big\|D_x^sP_{+HI}\big(P_{+hi}e^{\frac{i}{2}F}P_{lo}(e^{-\frac{i}{2}F}u)\big)\big\|_{L^p_TL^q_x}
\\ & \quad
+\big\|D_x^sP_{+HI}\big(P_{+HI}e^{\frac{i}{2}F}\partial_xP_{-hi}e^{-\frac{i}{2}F}\big)\big\|_{L^p_TL^q_x} \\
& :=I+II+III.
\end{split}
\end{displaymath}
Estimates \eqref{bourgstrich.3} and \eqref{lemma1.1} yield
\begin{equation} \label{apriori u.4}
I \lesssim (1+\|u\|_{L^{\infty}_TL^2_x})\|J^s_xw\|_{L^p_TL^q_x} \lesssim
(1+\|u\|_{L^{\infty}_TL^2_x})\|w\|_{Y^s_T}.
\end{equation}
On the other hand the fractional Leibniz rule (cf. Lemma \ref{leibrule}), H\"older's inequality in time and the Sobolev embedding imply that
\begin{equation} \label{apriori u.5}
\begin{split}
II &\lesssim
\|D^s_xP_{+hi}e^{\frac{i}{2}F}\|_{L^p_TL^q_x}\|P_{+lo}\big(ue^{-\frac{i}{2}F}\big)\|_{L^{\infty}_{T,x}}
\\ & \quad +\|P_{+hi}e^{\frac{i}{2}F}\|_{L^{\infty}_{T,x}}
\|D_x^sP_{+lo}\big(ue^{-\frac{i}{2}F}\big)\|_{L^p_TL^q_x}
\\ & \lesssim
\|\partial_xP_{+hi}e^{\frac{i}{2}F}\|_{L^p_TL^2_x}\|P_{+lo}\big(ue^{-\frac{i}{2}F}\big)\|_{L^{\infty}_{T}L^2_x}
\\ & \lesssim
T^{\frac1p}\|u\|_{L^{\infty}_TL^2_x}^2.
\end{split}
\end{equation}
Finally estimate \eqref{lemma2.1} with $\alpha_1=\alpha_2=(1+s)/2$ and $q_1=q_2=q$, H\"older's inequality in time and the Sobolev embedding lead to
\begin{equation} \label{apriori u.6}
\begin{split}
III &\lesssim
\|D^{(1+s)/2}_xP_{+HI}e^{\frac{i}{2}F}\|_{L^{2p}_TL^{2q}_x}\|D^{(1+s)/2}_xP_{-hi}e^{-\frac{i}{2}F}\|_{L^{2p}_TL^{2q}_x}
\\ & \lesssim T^{\frac1p}
\|D^{1+\frac{s}{2}-\frac{1}{2q}}_xP_{+HI}e^{\frac{i}{2}F}\|_{L^{\infty}_TL^{2}_x}
\|D^{1+\frac{s}{2}-\frac{1}{2q}}_xP_{-hi}e^{-\frac{i}{2}F}\|_{L^{\infty}_TL^{2}_x}
\\ & \lesssim T^{\frac1p}
\|\partial_xP_{+HI}e^{\frac{i}{2}F}\|_{L^{\infty}_TL^{2}_x}
\|\partial_xP_{-hi}e^{-\frac{i}{2}F}\|_{L^{\infty}_TL^{2}_x}
\\ &\lesssim T^{\frac1p}\|u\|_{L^{\infty}_TL^2_x}^2,
\end{split}
\end{equation}
since $0\le s \le \frac1q$.
Therefore, we deduce gathering \eqref{apriori u.4}--\eqref{apriori
u.6} that
\begin{equation} \label{apriori u.7}
\|D_x^sP_{+HI}u\|_{L^p_TL^q_x} \lesssim
\big(1+\|u\|_{L^{\infty}_TL^2_x}\big)\big(\|w\|_{Y^s_T}+T^{\frac1p}\|u\|_{L^{\infty}_TL^2_x}^2\big).
\end{equation}

Next we turn to the first term on the right-hand side of
\eqref{apriori u.3} and consider the integral equation satisfied by
$P_{LO}u$,
\begin{equation} \label{apriori u.8b}
P_{LO}u=U(t)P_{LO}u_0+\int_0^tU(t-\tau)P_{LO}\partial_x(u^2)(\tau)d\tau.
\end{equation}
First, observe that
\begin{displaymath}
\|P_{LO}u\|_{L^p_TL^q_x} \lesssim T^{\frac1p}\|P_{LO}u\|_{L^{\infty}_TL^2_x}
\end{displaymath}
Then, we deduce from \eqref{apriori u.8b}, using the fact that $U$ is a
unitary group in $L^2$ and Bernstein's inequality, that
\begin{equation} \label{apriori u.8}
\begin{split}
\|P_{LO}u\|_{L^p_TL^q_x} &\lesssim
T^{\frac1p}\|u_0\|_{L^2_x}+T^{1+\frac1p}\|\partial_xP_{LO}(u^2)\|_{L^{\infty}_TL^2_x}
\\ & \lesssim
T^{\frac1p}\|u_0\|_{L^2_x}+T^{1+\frac1p}\|P_{LO}(u^2)\|_{L^{\infty}_TL^1_x}
\\ & \lesssim \|u_0\|_{L^2_x}+\|u\|_{L^{\infty}_TL^2_x}^2,
\end{split}
\end{equation}
since $0 \le T \le 1$.

Thus, estimate \eqref{apriori u.2} follows combining \eqref{apriori
u.3}, \eqref{apriori u.7} and \eqref{apriori u.8}. This concludes
the proof of Proposition \ref{apriori u}.
\end{proof}

\subsection{Bilinear estimates}
The aim of this subsection is to derive the following estimate on
$\|w\|_{Y^s_T}$ :
\begin{proposition} \label{apriori w}
Let $0<T\le 1$, $0\le s \le \frac12$ and $u$ be a solution to
\eqref{BO} on the time interval $[0,T]$. Then it holds that
\begin{equation} \label{apriori w.1}
\begin{split}
\|w\|_{Y^s_T} &\lesssim \big(1+\|u_0\|_{L^2}
\big)\|u_0\|_{H^s}+\|u\|_{L^4_{x,T}}^2
\\ & \quad
+\|w\|_{X^{s,1/2}_T}\big(\|u\|_{L^{\infty}_TL^2_x}+\|u\|_{L^4_{x,T}}+\|u\|_{X^{-1,1}_T}\big).
\end{split}
\end{equation}
\end{proposition}

The main tool to  prove Proposition \ref{apriori w} is  the following crucial
bilinear estimates.
\begin{proposition} \label{bilincrit}
Let $s \ge 0$. Then we have that
\begin{equation} \label{bilincrit1}
\begin{split}
\|\partial_xP_{+hi}\big(\partial_x^{-1}wP_-\partial_x&u\big)\|_{X^{s,-\frac12}}
\\& \lesssim
\|w\|_{X^{s,\frac12}}\big(\|u\|_{L^2_{x,t}}+\|u\|_{L^4_{x,t}}+\|u\|_{X^{-1,1}}\big),
\end{split}
\end{equation}
and
\begin{equation} \label{bilincrit1b}
\begin{split}
\|\partial_xP_{+hi}\big(\partial_x^{-1}wP_-\partial_x&u\big)\|_{\widetilde{Z}^{s,-1}}
\\ & \lesssim
\|w\|_{X^{s,\frac12}}\big(\|u\|_{L^2_{x,t}}+\|u\|_{L^4_{x,t}}+\|u\|_{X^{-1,1}}\big).
\end{split}
\end{equation}
\end{proposition}

\begin{remark} Note that $\partial_x^{-1}w$ is well defined since $w$ is localized in high frequencies.
\end{remark}

\begin{proof} We will only give the proof in the case of $s=0$, since the case $s>0$ can be deduced by using similar arguments. By duality to prove \eqref{bilincrit1} is equivalent to prove that
\begin{equation} \label{bilincrit2}
\big|I\big| \lesssim \|h\|_{L^2_{x,t}}\|w\|_{X^{0,\frac12}}\big(\|u\|_{L^2_{x,t}}+\|u\|_{L^4_{x,t}}+\|u\|_{X^{-1,1}}\big),
\end{equation}
where
\begin{equation} \label{bilincrit3}
I=\int_{\mathcal{D}}\frac{\xi}
{\langle \sigma\rangle^{\frac12}}\widehat{h} (\xi,\tau)\xi_1^{-1}\widehat{w}(\xi_1,\tau_1)
\xi_2 \widehat{u}(\xi_2,\tau_2)d\nu,
\end{equation}
\begin{equation} \label{bilincrit3bb}
d\nu=d\xi d\xi_1 d\tau d\tau_1, \quad \xi_2=\xi-\xi_1, \quad \tau_2=\tau-\tau_1, \quad \sigma_i=\tau_i+\xi_i|\xi_i|, \ i=1,2,
\end{equation}
and
\begin{equation} \label{bilincrit3bbb}
\mathcal{D}=\big\{(\xi,\xi_1,\tau,\tau_1) \in \mathbb R^4 \ | \ \xi \ge 1, \ \xi_1 \ge 1 \ \text{and} \ \xi_2 \le 0 \big\}.
\end{equation}
Observe that we always have in $\mathcal{D}$ that
\begin{equation} \label{bilincrit3b}
\xi_1 \ge \xi \ge 1 \quad \text{and} \quad \xi_1 \ge |\xi_2|.
\end{equation}
In the case where $|\xi_2| \le 1$, we have by using H\"older's inequality and estimate \eqref{bourgstrich.2} that
\begin{displaymath}
\begin{split}
\big| I \big| &\lesssim \int_{\mathbb R^4}\frac{|\widehat{h}|}{\langle \sigma\rangle^{\frac12}}
|\widehat{w}(\xi_1,\tau_1)||\widehat{u}(\xi_2,\tau_2)|d\nu \\ &
\lesssim \bigl\|\Big(\frac{|\widehat{h}|}{\langle \sigma \rangle^{\frac12}}\Big)^{\vee}\bigr\|_{L^4_{x,t}}
\|(|\widehat{w}|)^{\vee}\|_{L^4_{x,t}}\|u\|_{L^2_{x,t}} \\ &
\lesssim \|h\|_{L^2_{x,t}}\|w\|_{X^{\frac38}}\|u\|_{L^2_{x,t}}.
\end{split}
\end{displaymath}
Then, from now on we will assume that $|\xi_2| \ge 1$ in $\mathcal{D}$.

By using a dyadic decomposition in space-frequency for the functions $ h$, $ w$ and $ u $ one can rewrite $ I $ as
\begin{equation} \label{bilincrit3bbbb}
I=\sum_{N,N_1,N_2} I_{N,N_1,N_2}
\end{equation}
with
\begin{displaymath}
I_{N, N_1,N_2}:= \int_{\mathcal{D}}\frac{\xi} {\langle
\sigma\rangle^{\frac12}}\widehat{P_Nh}
(\xi,\tau)\xi_1^{-1}\widehat{P_{N_1}w}(\xi_1,\tau_1) \xi_2
\widehat{P_{N_2}u}(\xi_2,\tau_2)d\nu,
\end{displaymath}
and the dyadic numbers $N, \ N_1$ and $N_2$ ranging from $1$ to $+\infty$.
Moreover, the resonance identity
\begin{equation} \label{bilincrit4}
\sigma_1+\sigma_2-\sigma=\xi_1^2+(\xi-\xi_1)|\xi-\xi_1|-\xi^2=-2\xi\xi_2
\end{equation}
holds in $\mathcal{D}$. Therefore, to calculate  $I_{N, N_1,N_2}$, we split the integration domain $\mathcal{D}$ in the following disjoint regions
\begin{equation} \label{bilincrit4b}
\begin{split}
\mathcal{A}_{N,N_2}&=\big\{(\xi,\xi_1,\tau,\tau_1) \in \mathcal{D} \ | \ |\sigma|\ge \frac 16  N N_2   \big\}, \\
\mathcal{B}_{N,N_2}&=\big\{(\xi,\xi_1,\tau,\tau_1) \in \mathcal{D} \ | \ |\sigma_1|\ge \frac 16  N N_2  \ , \ |\sigma|< \frac16 N N_2 \big\}, \\
\mathcal{C}_{N,N_2}&=\big\{(\xi,\xi_1,\tau,\tau_1) \in \mathcal{D} \ | \ |\sigma|< \frac16 N N_2 , \ |\sigma_1|< \frac16 N N_2 \ , |\sigma_2|\ge
 \frac16 NN_2 \big\},
\end{split}
\end{equation}
and denote by $I^{\mathcal{A}_{N,N_2}}_{N, N_1,N_2}$, $I^{\mathcal{B}_{N,N_2}}_{N, N_1,N_2}$, $I^{\mathcal{C}_{N,N_2}}_{N, N_1,N_2}$ the restriction of $I_{N, N_1,N_2}$ to each of these regions. Then, it follows that
$$
I_{N, N_1,N_2}=I^{\mathcal{A}_{N,N_2}}_{N, N_1,N_2}+I^{\mathcal{B}_{N,N_2}}_{N, N_1,N_2}+I^{\mathcal{C}_{N,N_2}}_{N, N_1,N_2}
$$
 and thus
\begin{equation} \label{bilincrit5}
\big|I\big| \le \big|I_{\mathcal{A}}\big|+\big|I_{\mathcal{B}}\big|+\big|I_{\mathcal{C}}\big|,
\end{equation}
where
$$
I_{\mathcal{A}}:=\sum_{N, N_1,N_2} I^{\mathcal{A}_{N,N_2}}_{N, N_1,N_2}, \; I_{\mathcal{B}}:=\sum_{N, N_1,N_2} I^{\mathcal{B}_{N,N_2}}_{N, N_1,N_2} \ \mbox{and } \ I_{\mathcal{C}}:=\sum_{N, N_1,N_2} I^{\mathcal{C}_{N,N_2}}_{N, N_1,N_2}\; .
 $$
Therefore,  it suffices to bound $\big|I_{\mathcal{A}}\big|$, $\big|I_{\mathcal{B}}\big|$ and $\big|I_{\mathcal{C}}\big|$. Note that one of the two following cases holds:
\begin{enumerate}
\item{} high-low interaction: $N_1 \sim N$ and $N_2 \le N_1$
\item{} high-high interaction: $N_1 \sim N_2$ and $N \le N_1$.
\end{enumerate}

\noindent \textit{Estimate for $\big|I_{\mathcal{A}}\big|$.} In the first case, we observe from the Cauchy-Schwarz inequality that
\begin{displaymath}
\begin{split}
\big|I_{\mathcal{A}}\big|&\sim\Big|\int_{\mathbb R^2}\widehat{h}\sum_{N_1}\sum_{j= 0}^{\frac{\ln(N_1)}{\ln(2)}}\phi_{N_1}\xi\langle \sigma\rangle^{-\frac12}\chi_{\{|\sigma|\ge\frac16{N_1}^22^{-j}\}}
\mathcal{F}\big(P_+\big(\partial_x^{-1}P_{N_1}wP_-\partial_xP_{2^{-j}N_1}u\big)\big)d\xi d\tau\Big|\\ &
\lesssim \|\widehat{h}\|_{L^2_{\xi,\tau}}
\Big\|\sum_{N_1}\sum_{j\ge 0} {N_1}^2 (N_1^2 2^{-j})^{-1}\phi_{N_1}
\Bigl| \mathcal{F}\big(P_+\big(\partial_x^{-1}P_{N_1}wP_-\partial_xP_{2^{-j}N_1}u\big)\big)\Bigr|
\Big\|_{L^2_{\xi,\tau}}.
\end{split}
\end{displaymath}
Then, the Plancherel identity and the triangular inequality imply that
\begin{displaymath}
\begin{split}
\big|I_{\mathcal{A}}\big| \lesssim  \|h \|_{L^2_{x,t}}  \sum_{j\ge 0} \Big(\sum_{N_1}   2^{j}
  \Bigr\|P_{N_1}\big(\partial_x^{-1}P_{N_1}wP_-\partial_xP_{2^{-j}N_1}u\big)\Bigl\|_{L^2_{x,t}}^2\Big)^{\frac12} .
\end{split}
\end{displaymath}
By using the H\"older and Bernstein inequalities, we deduce that
\begin{equation} \label{bilincrit6}
\begin{split}
\big|I_{\mathcal{A}}\big| &\lesssim  \|h \|_{L^2_{x,t}}
\sum_{j\ge0} \Big(\sum_{N_1} 2^{-j}
\|P_{N_1}w\|_{L^4_{x,t}}^2\|P_{2^{-j}N_1}u\|_{L^4_{x,t}}^2\Big)^{\frac12}\\
&\lesssim \|h \|_{L^2_{x,t}}
 \Big(\sum_{N}
\|P_{N_1}w\|_{L^4_{x,t}}^2\Big)^{\frac12}\|u\|_{L^4_{x,t}}.
\end{split}
\end{equation}
In the second case, it follows using the same strategy as in the
first case, that
\begin{displaymath}
\begin{split}
\big|I_{\mathcal{A}}\big| &\lesssim  \|h \|_{L^2_{x,t}} \\ & \times \sum_{j\ge0} \Big(\sum_{N_1}  (2^{-j}N_1)^2 (2^{-j}N_1 N_1)^{-1}
  \Bigr\|P_{2^{-j}N_1}\big(\partial_x^{-1}P_{N_1}wP_-\partial_xP_{N_1}u\big)\Bigl\|_{L^2_{x,t}}^2\Big)^{\frac12} \;,
\end{split}
\end{displaymath}
which implies using the H\"older and Bernstein inequalities
\begin{equation} \label{bilincrit7}
\begin{split}
\big|I_{\mathcal{A}}\big| &\lesssim  \|h \|_{L^2_{x,t}}
\sum_{j\ge0} \Big(\sum_{N_1} 2^{-j}
\|P_{N_1}w\|_{L^4_{x,t}}^2\|P_{N_1}u\|_{L^4_{x,t}}^2\Big)^{\frac12}\\
&\lesssim \|h \|_{L^2_{x,t}}
 \Big(\sum_{N_1}
\|P_{N_1}w\|_{L^4_{x,t}}^2\Big)^{\frac12}\|u\|_{L^4_{x,t}}.
\end{split}
\end{equation}
Therefore, we deduce gathering \eqref{bilincrit6}--\eqref{bilincrit7} and using estimate \eqref{bourgstrich.2} that
\begin{equation} \label{bilincrit9}
\big|I_{\mathcal{A}}\big| \le  \|h \|_{L^2_{x,t}}\|w\|_{X^{0,\frac38}}\|u\|_{L^4_{x,t}}.
\end{equation}

\noindent \textit{Estimate for $\big|I_{\mathcal{B}}\big|$.} By
using again the triangular and the Cauchy-Schwarz inequalities, we
have in the first case that
\begin{displaymath}
\begin{split}
\big|&I_{\mathcal{B}}\big|  \le  \|w \|_{X^{0,\frac12}} \\ & \times \sum_{j \ge 0}
\left(\sum_{N_1}  N_1^{-2} (N_1 2^{-j}N_1)^{-1}
\Bigr\|P_{N_1}\big(\partial_xP_{+hi}P_{N_1}\Big(\frac{\widehat{h}}{\langle \sigma \rangle^{\frac12}}\Big)^{\vee}P_+\partial_xP_{2^{-j}N_1}\tilde{u}\big)\Bigl\|_{L^2_{x,t}}^2\right)^{\frac12},
\end{split}
\end{displaymath}
where $\tilde{u}(x,t)=u(-x,-t)$.
Thus it follows from the Bernstein and H\"older inequalities that
\begin{equation} \label{bilincrit10}
\begin{split}
\big|I_{\mathcal{B}}\big| &\lesssim  \|w \|_{X^{0,\frac12}}
\sum_{j\ge0} \Big(\sum_{N_1} 2^{-j}
\bigl\|P_{N_1}\Big(\frac{\widehat{h}}{\langle \sigma \rangle^{\frac12}}\Big)^{\vee}\bigr\|_{L^4_{x,t}}^2\|P_{2^{-j}N_1}u\|_{L^4_{x,t}}^2\Big)^{\frac12}\\
&\lesssim \|w \|_{X^{0,\frac12}}
 \Big(\sum_{N_1}
\bigl\|P_{N_1}\Big(\frac{\widehat{h}}{\langle \sigma \rangle^{\frac12}}\Big)^{\vee}\bigr\|_{L^4_{x,t}}^2\Big)^{\frac12}\|u\|_{L^4_{x,t}}.
\end{split}
\end{equation}
In the second case, we  bound $\big|I_{\mathcal{B}}\big|$ as follows,
\begin{displaymath}
\begin{split}
\big|&I_{\mathcal{B}}\big|  \le  \|w \|_{X^{0,\frac12}} \\ & \times \sum_{j \ge 0}
\left(\sum_{N_1}  N_1^{-2} (2^{-j}N_1 N_1)^{-1}
\Bigr\|P_{N_1}\big(\partial_xP_{+hi}P_{2^{-j}N_1}\Big(\frac{\widehat{h}}{\langle \sigma \rangle^{\frac12}}\Big)^{\vee}P_+\partial_xP_{N_1}\tilde{u}\big)\Bigl\|_{L^2_{x,t}}^2\right)^{\frac12},
\end{split}
\end{displaymath}
so that
\begin{equation} \label{bilincrit11}
\begin{split}
\big|I_{\mathcal{B}}\big| &\lesssim  \|w \|_{X^{0,\frac12}}
\sum_{j\ge0} \Big(\sum_{N_1} 2^{-j}
\bigl\|P_{2^{-j}N_1}P_{+hi}\Big(\frac{\widehat{h}}{\langle \sigma \rangle^{\frac12}}\Big)^{\vee}\bigr\|_{L^4_{x,t}}^2\|P_{N_1}u\|_{L^4_{x,t}}^2\Big)^{\frac12}\\
&\lesssim \|w \|_{X^{0,\frac12}}\sum_{j \ge 0}2^{-\frac{j}2}
\Big(\sum_{N_1}
\bigl\|P_{2^{-j}N_1}P_{+hi}\Big(\frac{\widehat{h}}{\langle \sigma \rangle^{\frac12}}\Big)^{\vee}\bigr\|_{L^4_{x,t}}^2\Big)^{\frac12}
\|u\|_{L^4_{x,t}} \\ &
\lesssim \|w \|_{X^{0,\frac12}}
\Big(\sum_{N_1}
\bigl\|P_{N_1}\Big(\frac{\widehat{h}}{\langle \sigma \rangle^{\frac12}}\Big)^{\vee}\bigr\|_{L^4_{x,t}}^2\Big)^{\frac12}
\|u\|_{L^4_{x,t}}.
\end{split}
\end{equation}
In conclusion, we obtain gathering \eqref{bilincrit10}--\eqref{bilincrit11} and using estimate \eqref{bourgstrich.2} that
\begin{equation} \label{bilincrit12}
\big|I_{\mathcal{B}}\big| \le  \|h \|_{L^2_{x,t}}\|w\|_{X^{0,\frac12}}\|u\|_{L^4_{x,t}}.
\end{equation}

\noindent \textit{Estimate for $\big|I_{\mathcal{C}}\big|$.} First observe that
\begin{equation} \label{bilincrit13}
\big|I_{\mathcal{C}}\big| \lesssim \int_{\widetilde{\mathcal{C}}}\frac{|\xi|}
{\langle \sigma\rangle^{\frac12}}|\widehat{h} (\xi,\tau)||\xi_1|^{-1}|\widehat{w}(\xi_1,\tau_1)|
 \frac{|\xi_2|^2}{\langle\sigma_2\rangle}\frac{\langle\sigma_2\rangle}{|\xi_2|}|\widehat{u}(\xi_2,\tau_2)|d\nu,
\end{equation}
where
\begin{displaymath}
\widetilde{\mathcal{C}}=\big\{(\xi,\xi_1,\tau,\tau_1) \in \mathcal{D} \ |
\ (\xi,\xi_1,\tau,\tau_1) \in \bigcup_{N,N_2}\mathcal{C}_{N,N_2} \big\}.
\end{displaymath}
Since $|\sigma_2|>|\sigma|$ and $|\sigma_2|>|\sigma_1|$ in $\widetilde{\mathcal{C}}$, it follows from \eqref{bilincrit4} that $|\sigma_2| \gtrsim |\xi\xi_2|$. Then,
\begin{equation}\label{ed}
|\xi\xi_1^{-1}\xi_2^2\langle\sigma_2\rangle^{-1}| \lesssim 1
\end{equation}
 holds in $\widetilde{\mathcal{C}}$, so that
\begin{equation} \label{bilincrit14b}
\begin{split}
\big|I_{\mathcal{C}}\big| &\lesssim \int_{\widetilde{\mathcal{C}}}\frac{|\widehat{h} (\xi,\tau)|}
{\langle \sigma\rangle^{\frac12}}|\widehat{w}(\xi_1,\tau_1)|
\frac{\langle\sigma_2\rangle}{|\xi_2|}|\widehat{u}(\xi_2,\tau_2)|d\nu  \\
& \lesssim \bigl\|\Big(\frac{|\widehat{h}|}{\langle \sigma \rangle^{\frac12}}\Big)^{\vee}\bigr\|_{L^4_{x,t}}
\|(|\widehat{w}|)^{\vee}\|_{L^4_{x,t}}\|u\|_{X^{-1,1}} \\ &
\lesssim \|h\|_{L^2_{x,t}}\|w\|_{X^{\frac38}}\|u\|_{X^{-1,1}}
\end{split}
\end{equation}
is deduced by using H\"older's inequality and estimate \eqref{bourgstrich.2}.

Therefore, estimates \eqref{bilincrit5}, \eqref{bilincrit9}, \eqref{bilincrit12} and \eqref{bilincrit14b} imply estimate \eqref{bilincrit2}, which concludes the proof of estimate \eqref{bilincrit1}.

To prove estimate \eqref{bilincrit1b}, we also proceed by duality. Then it is sufficient to show that
\begin{equation} \label{bilincrit14}
\big|J\big| \lesssim
\big(\sum_{N}\|g_N\|_{L^2_{\xi}L^{\infty}_{\tau}}^2\big)^{\frac12}\|w\|_{X^{0,\frac12}}
\big(\|u\|_{L^2_{x,t}}+\|u\|_{L^4_{x,t}}+\|u\|_{X^{-1,1}}\big),
\end{equation}
where
\begin{displaymath}
J=\sum_N\int_{\mathcal{D}}\frac{\xi} {\langle \sigma\rangle}
g_N(\xi,\tau)\phi_N(\xi)\xi_1^{-1}\widehat{w}(\xi_1,\tau_1) \xi_2
\widehat{u}(\xi_2,\tau_2)d\nu,
\end{displaymath}
and $d\nu$ and $\mathcal{D}$ are defined in \eqref{bilincrit3bb} and
\eqref{bilincrit3bbb}. As in the case of $I$, we can also assume that $|\xi_2|\ge 1$. By using dyadic decompositions as in
\eqref{bilincrit3bbbb}, $J$ can be rewritten as
\begin{displaymath}
J=\sum_{N,N_1,N_2} J_{N,N_1,N_2},
\end{displaymath}
where
\begin{displaymath}
J_{N, N_1,N_2}:= \int_{\mathcal{D}}\frac{\xi} {\langle
\sigma\rangle}\phi_N(\xi)g_N
(\xi,\tau)\xi_1^{-1}\widehat{P_{N_1}w}(\xi_1,\tau_1) \xi_2
\widehat{P_{N_2}u }(\xi_2,\tau_2)d\nu,
\end{displaymath}
and the dyadic numbers $N$, $N_1$ and $N_2$ range from $1$ to $+\infty$.
Moreover, we will denote by $J^{\mathcal{A}_{N,N_2}}_{N, N_1,N_2}$, $J^{\mathcal{B}_{N,N_2}}_{N, N_1,N_2}$, $J^{\mathcal{C}_{N,N_2}}_{N, N_1,N_2}$ the restriction of $J_{N, N_1,N_2}$ to the regions $\mathcal{A}_{N,N_2}$, $\mathcal{B}_{N,N_2}$ and $\mathcal{C}_{N,N_2}$ defined in \eqref{bilincrit4}. Then, it follows that
\begin{equation} \label{bilincrit15}
\big|J\big| \le \big|J_{\mathcal{A}}\big|+\big|J_{\mathcal{B}}\big|+\big|J_{\mathcal{C}}\big|,
\end{equation}
where
$$
J_{\mathcal{A}}:=\sum_{N, N_1,N_2} J^{\mathcal{A}_{N,N_2}}_{N, N_1,N_2}, \; J_{\mathcal{B}}:=\sum_{N, N_1,N_2} J^{\mathcal{B}_{N,N_2}}_{N, N_1,N_2} \ \mbox{and } \ J_{\mathcal{C}}:=\sum_{N, N_1,N_2} J^{\mathcal{C}_{N,N_2}}_{N, N_1,N_2}\; ,
 $$
so that  it suffices to estimate $\big|J_{\mathcal{A}}\big|$, $\big|J_{\mathcal{B}}\big|$ and $\big|J_{\mathcal{C}}\big|$.

\noindent \textit{Estimate for $\big|J_{\mathcal{A}}\big|$.} To
estimate $\big|J_{\mathcal{A}}\big|$, we divide each region
$\mathcal{A}_{N,N_2}$ into disjoint subregions
\begin{displaymath}
\mathcal{A}_{N,N_2}^q=\big\{(\xi,\xi_1,\tau,\tau_1) \in
\mathcal{A}_{N,N_2} \ | \ 2^{q-3}NN_2 \le |\sigma| < 2^{q-2}NN_2
\big\},
\end{displaymath}
for $q \in \mathbb Z_+$. Thus if $J^{\mathcal{A}_{N,N_2}^q}_{N,
N_1,N_2}$ denote the restriction of $J^{\mathcal{A}_{N,N_2}}_{N,
N_1,N_2}$ to each of these regions, we have that
$J_{\mathcal{A}}=\sum_{q \ge 0}\sum_{N, N_1,N_2}
J^{\mathcal{A}_{N,N_2}^q}_{N, N_1,N_2}$. In the case of high-low
interactions, we deduce by using the Plancherel identity
Cauchy-Schwarz and Minkowski inequalities that
\begin{displaymath}
\begin{split}
|J_{\mathcal{A}}|\le \sum_{q \ge 0}\sum_{N_1}\sum_{N_2 \le N_1}&\|g_{N_1}
\chi_{\{|\sigma|\sim2^qN_1N_2\}}\|_{L^2_{\xi,\tau}} \\
&\times(2^qN_1N_2)^{-1}N_1
\Bigr\|\partial_x^{-1}P_{N_1}wP_-\partial_xP_{N_2}u\Bigl\|_{L^2_{x,t}}.
\end{split}
\end{displaymath}
Moreover, we get from H\"older's inequality
\begin{displaymath}
\|g_{N_1} \chi_{\{|\sigma|\sim 2^qN_1N_2\}}\|_{L^2_{\xi,\tau}} \lesssim
(2^qNN_2)^{\frac12}\|g_{N_1} \|_{L^2_{\xi}L^{\infty}_{\tau}},
\end{displaymath}
so that, the Cauchy-Schwarz inequality yields
\begin{equation} \label{bilincrit16}
\begin{split}
|J_{\mathcal{A}}| &\lesssim \sum_{N_1}\sum_{N_2 \le
N_1}(N_2N_1^{-1})^{\frac12}\|g_{N_1}
\|_{L^2_{\xi}L^{\infty}_{\tau}}\|P_{N_1}w\|_{L^4_{x,t}}
\|P_{N_2}u\|_{L^4_{x,t}} \\
& \lesssim \|u\|_{L^4_{x,t}}\sum_{N_1}
\|g_{N_1}\|_{L^2_{\xi}L^{\infty}_{\tau}}\|P_{N_1}w\|_{L^4_{x,t}}
\\
&\lesssim \big(\sum_{N_1}\|g_{N_1}\|_{L^2_{\xi}L^{\infty}_{\tau}}^2\big)^{\frac12}
\|w\|_{\tilde{L}^4_{x,t}}\|u\|_{L^4_{x,t}}.
\end{split}
\end{equation}

In the high-high interaction case, it follows from the Minkowski and
Cauchy-Schwarz inequalities that
\begin{displaymath}
\begin{split}
|J_{\mathcal{A}}|\le \sum_{q \ge 0}\sum_{N_1}\sum_{N \le N_1}&\|g_N
\chi_{\{|\sigma|\sim2^qNN_1\}}\|_{L^2_{\xi,\tau}} \\
&\times(2^qNN_1)^{-1}N
\Bigr\|\partial_x^{-1}P_{N_1}wP_-\partial_xP_{N_1}u\Bigl\|_{L^2_{x,t}}.
\end{split}
\end{displaymath}
Moreover, we deduce from H\"older's inequality that
\begin{displaymath}
\|g_N \chi_{\{|\sigma|\sim2^qNN_1\}}\|_{L^2_{\xi,\tau}} \lesssim
(2^qNN_1)^{\frac12}\|g_N \|_{L^2_{\xi}L^{\infty}_{\tau}}.
\end{displaymath}
Then, the Cauchy-Schwarz inequality implies that
\begin{equation} \label{bilincrit17}
\begin{split}
|J_{\mathcal{A}}| &\lesssim \sum_{j \ge 0}\sum_{N_1}(N_1^{-1}2^{-j}N_1)^{\frac12}\|g_{2^{-j}N_1}
\|_{L^2_{\xi}L^{\infty}_{\tau}}\|P_{N_1}w\|_{L^4_{x,t}}
\|P_{N_1}u\|_{L^4_{x,t}}
\\
& \lesssim \sum_{j \ge 0}2^{-\frac{j}{2}}
\big(\sum_{N_1}\|g_{2^{-j}N_1}\|_{L^2_{\xi}L^{\infty}_{\tau}}^2\big)^{\frac12}
\big(\sum_{N_1}\|P_{N_1}w\|_{L^4_{x,t}}^2\big)^{\frac12}\|u\|_{L^4_{x,t}} \\
& \lesssim
\big(\sum_{N_1}\|g_{N_1}\|_{L^2_{\xi}L^{\infty}_{\tau}}^2\big)^{\frac12}
\|w\|_{\tilde{L}^4_{x,t}}\|u\|_{L^4_{x,t}}.
\end{split}
\end{equation}
Then estimates \eqref{bourgstrich.2}, \eqref{bilincrit16} and
\eqref{bilincrit17} yield
\begin{equation} \label{bilincrit19}
|J_{\mathcal{A}}| \lesssim
\big(\sum_N\|g_N
\|_{L^2_{\xi}L^{\infty}_{\tau}}^2\big)^{\frac12}
\|w\|_{X^{0,\frac38}}\|u\|_{L^4_{x,t}}.
\end{equation}

\noindent \textit{Estimate for $\big|J_{\mathcal{B}}\big|$ and $\big|J_{\mathcal{C}}\big|$.}
Arguing as in the proof of \eqref{bilincrit1}, it is deduced that
\begin{displaymath}
\big|J_{\mathcal{B}}\big|+\big|J_{\mathcal{C}}\big|
\lesssim
\Big(\bigl\|\Big(\frac{g}{\langle \sigma \rangle}\Big)^{\vee}\bigr\|_{\widetilde{L}^4_{x,t}}
+\bigl\|\Big(\frac{|g|}{\langle \sigma \rangle}\Big)^{\vee}\bigr\|_{\widetilde{L}^4_{x,t}}\Big)
\|w\|_{X^{0,\frac12}}
\big(\|u\|_{L^4_{x,t}}+\|u\|_{X^{-1,1}}\big),
\end{displaymath}
where $g=\sum_N\phi_Ng_N$.
Moreover, estimate \eqref{bourgstrich.2} and H\"older's inequality imply \begin{displaymath}
\begin{split}
\bigl\|\Big(\frac{g}{\langle \sigma \rangle}\Big)^{\vee}\bigr\|_{\widetilde{L}^4_{x,t}}
+\bigl\|\Big(\frac{|g|}{\langle \sigma \rangle}\Big)^{\vee}\bigr\|_{\widetilde{L}^4_{x,t}}
&\lesssim \|\langle \sigma \rangle^{-\frac58}\sum_N\phi_Ng_N\|_{L^2_{\xi,\tau}} \\ &
\lesssim \big(\sum_N\|\langle \sigma \rangle^{-\frac58}g_N\|_{L^2_{\xi,\tau}}^2\big)^{\frac12}\\&
\lesssim \big(\sum_N\|g_N\|_{L^2_{\xi}L^{\infty}_{\tau}}^2\big)^{\frac12},
\end{split}
\end{displaymath}
so that
\begin{equation} \label{bilincrit20}
\big|J_{\mathcal{B}}\big|+\big|J_{\mathcal{C}}\big|
\lesssim \big(\sum_N\|g_N\|_{L^2_{\xi}L^{\infty}_{\tau}}^2\big)^{\frac12}
\|w\|_{X^{0,\frac12}}
\big(\|u\|_{L^4_{x,t}}+\|u\|_{X^{-1,1}}\big).
\end{equation}

Finally \eqref{bilincrit15}, \eqref{bilincrit19} and
\eqref{bilincrit20} imply \eqref{bilincrit14}, which concludes the
proof of estimate \eqref{bilincrit1b}.
\end{proof}

\begin{lemma} \label{lemma3}
Let $0<T\le 1$, $s \ge 0$ , $u_1$, $u_2\in L^\infty(\R; L^2(\R)) \cap L^4(\R^2)  $
supported in the time interval $[-2T,2T]$, and $F_1$, $F_2$ be some spatial primitive of respectively $ u_1 $ and $ u_2 $. Then
\begin{equation} \label{lemma3.1}
\begin{split}
&\big\|\partial_xP_{+hi}\big(P_{lo}e^{-\frac{i}{2}F_1}
P_-\partial_xu_1\big)\big\|_{\tilde{Z}^{s,-1}} \\&
+\big\|\partial_xP_{+hi}\big(P_{lo}e^{-\frac{i}{2}F_1}
P_-\partial_xu_1\big)\big\|_{X^{s,-\frac12}}
\lesssim \|u_1\|_{L^4_{x,t}}^2,
\end{split}
\end{equation}
and
\begin{equation} \label{lemma3.1b}
\begin{split}
&\big\|\partial_xP_{+hi}\big(P_{lo}\big(e^{-\frac{i}{2}F_1}-e^{-\frac{i}{2}F_2}\big)
P_-\partial_xu_2\big)\big\|_{\tilde{Z}^{s,-1}} \\&
+\big\|\partial_xP_{+hi}\big(P_{lo}\big(e^{-\frac{i}{2}F_1}-e^{-\frac{i}{2}F_2}\big)
P_-\partial_xu_2\big)\big\|_{X^{s,-\frac12}} \\ &
\quad \quad \quad \lesssim \Big(\|u_1-u_2\|_{L^{\infty}_tL^2_x}+\|e^{-\frac{i}{2}F_1}
-e^{-\frac{i}{2}F_2}\|_{L^{\infty}_{x,t}} \|u_2\|_{L^{\infty}_tL^2_x}\Big)\|u_2\|_{L^4_{x,t}}.
\end{split}
\end{equation}
\end{lemma}

\begin{proof}
We deduce from the Cauchy-Schwarz inequality, the Sobolev embedding
$\|f\|_{H^{-\frac12+\epsilon}_t} \lesssim \|f\|_{L^{1+\epsilon'}_t}$
with $1+\epsilon'=\frac{1}{1-\epsilon}$, and the Minkowski
inequality that
\begin{equation} \label{lemma3.2}
\begin{split}
\|f\|_{\tilde{Z}^{s,-1}}+\|f\|_{X^{s,-\frac12}} &\lesssim
\|f\|_{X^{s,-\frac12+\epsilon}}
=\Big\|\big\|\big(J^s_xU(-t)f\big)^{\wedge_x}(\xi)\big\|_{H^{-\frac12+\epsilon}_t}\Big\|_{L^2_{\xi}}
 \\ & \lesssim
\Big\|\|\big(J^s_xU(-t)f\big)^{\wedge_x}(\xi)\big\|_{L^{1+\epsilon'}_t}
\Big\|_{L^2_{\xi}} \lesssim \|f\|_{L^{1+\epsilon'}_tH^s_x}.
\end{split}
\end{equation}
On the other hand, it follows from the frequency localization that
\begin{displaymath}
\partial_xP_{+hi}\big(P_{lo}e^{-\frac{i}{2}F}P_-\partial_xu\big)
=\partial_xP_{+LO}\big(P_{lo}e^{-\frac{i}{2}F}P_{-LO}\partial_xu\big).
\end{displaymath}
Therefore, by using \eqref{lemma3.2}, Bernstein's inequalities and
estimate \eqref{lemma2.1}, we can bound the left-hand side of
\eqref{lemma3.1} by
\begin{equation}\label{zq}
\big\|P_{+LO}\Big(P_{lo}e^{-\frac{i}{2}F}P_{-LO}\partial_xu\Big)\big\|_{L^{1+\epsilon'}_tL^2_x}
\lesssim T^{\gamma}
\|\partial_xe^{-\frac{i}{2}F}\|_{L^4_{x,t}}\|u\|_{L^4_{x,t}},
\end{equation}
with $\frac{1}{\gamma}=\frac12-\epsilon'$, which concludes the proof
of estimate \eqref{lemma3.1} recalling that $\partial_xF=u$ and $0 < T \le
1$. Estimate \eqref{lemma3.1b} can be proved exactly as above recalling \eqref{lemma1.6}.
\end{proof}

A proof of Proposition \ref{apriori w} is now in sight.
\begin{proof}[Proof of Proposition \ref{apriori w}]
Let $0 \le s \le \frac12$, $0<T\le 1$ and let $\tilde{u}$ and
$\tilde{w}$ be extensions of $u$ and $w$ such that
$\|\tilde{u}\|_{X^{-1,1}} \le 2\|u\|_{X^{-1,1}_T}$ and
$\|\tilde{w}\|_{X^{s,1/2}} \le 2\|w\|_{X^{s,1/2}_T}$. By the Duhamel principle,
the integral formulation associated to \eqref{gauge2} reads
\begin{displaymath}
\begin{split}
w(t)&=\eta(t)w(0)-\eta(t)\int_0^tU(t-t')\partial_xP_{+hi}\big(\eta_T\partial_x^{-1}\tilde{w}P_-\big(\eta_T\partial_xu\big)\big)(t')dt'
\\ & \quad
-\eta(t)\int_0^t
\partial_xP_{+hi}\big(P_{lo}\big(\eta_Te^{-\frac{i}{2}\tilde{F}}\big)P_-\big(\eta_T\partial_x\tilde{u}\big)\big)(t')dt',
\end{split}
\end{displaymath}
for $0<t \le T\le 1$. Therefore, we deduce gathering estimates
\eqref{prop1.1.2}, \eqref{prop1.2.2}, \eqref{bilincrit1},
\eqref{bilincrit1b} and \eqref{lemma3.1} that
\begin{displaymath}
\|w\|_{Y^s_T} \lesssim \|w(0)\|_{H^s}+\|u\|_{L^4_{x,T}}^2
+\|w\|_{X^{s1/2}_T}\|\big(\|u\|_{L^{\infty}_TL^2_x}+\|u\|_{L^4_{x,T}}+\|u\|_{X^{-1,1}_T}\big).
\end{displaymath}
This concludes the proof of estimate \eqref{apriori w.1}, since
\begin{equation}\label{estimatew0}
\|w(0)\|_{H^s} \lesssim
\big\|J^s_x\big(e^{-\frac{i}{2}F(\cdot,0)}u_0\big)\big\|_{L^2}
\lesssim \big(1+\|u_0\|_{L^2}\big)\|u_0\|_{H^s},
\end{equation}
follows from estimate \eqref{lemma1.1} and the fact that $0 \le s
\le \frac12$.
\end{proof}

\section{Proof of Theorem \ref{theo1}}
First it is worth noticing that we can always assume that we deal with data that have  small $ L^2(\R) $-norm.
Indeed,  if $u$ is a solution to the IVP \eqref{BO} on the time
interval $[0,T]$ then,  for every $0<\lambda<\infty $,
$u_{\lambda}(x,t)=\lambda u(\lambda x,\lambda^2t)$ is also a
solution to the equation in \eqref{BO} on the time interval
$[0,\lambda^{-2}T]$ with initial data $u_{0,\lambda}=\lambda
u_{0}(\lambda \cdot)$.  For  $\varepsilon>0 $ let us denote by $ B_\varepsilon $ the ball of $ L^2(\R) ,$ centered at the origin with radius $ \varepsilon $. Since $\|u_{\lambda}(\cdot,0)\|_{L^2} =\lambda^{\frac12}\|u_0\|_{L^2}$,  we see that we can  force $u_{0,\lambda}$  to belong to $ B_\epsilon$ by
choosing $\lambda \sim \min( \varepsilon^2\|u_0\|_{L^2}^{-2},1) $.
Therefore the existence and uniqueness of a solution of \eqref{BO} on the time interval $ [0,1] $ for small $ L^2(\R) $-initial data will ensure
  the existence of a unique solution $u$ to \eqref{BO} for arbitrary large $L^2(\R) $-initial data  on the time
interval $T\sim \lambda^2 \sim \min( \|u_0\|_{L^2}^{-4},1)$. Using the conservation of the $ L^2(\R) $-norm,  this will lead
to global well-posedness in $ L^2(\R) $.

\subsection{Uniform bound for small initial data} \label{Exist}
First, we begin by deriving  \textit{a priori} estimates on
smooth solutions associated to  initial data $u_0\in H^s(\R)$ that is small in $L^2(\R) $ . It is known from the
classical well-posedness theory (cf. \cite{Io}) that such an initial data gives rise to a global solution $u \in C(\mathbb R; H^{\infty}(\mathbb R))$ to the
Cauchy problem \eqref{BO}. Setting for $ 0<T\le 1 $,
\begin{equation} \label{theo1.2}
N_T^s(u):=\max\Bigl(\|u\|_{L^{\infty}_T H^s_x},\,\|J_x^s u\|_{L^4_{x,T}},
\|w\|_{X^{s,\frac12}_T} \Bigr) ,
 \end{equation}
 it follows from the smoothness of $ u$ that
$ T\mapsto N_T^s(u) $ is continuous and non decreasing on $ \mathbb R_+^* $. Moreover,  from \eqref{gauge2}, the linear estimate \eqref{prop1.2.2}, \eqref{estimatew0}
and \eqref{apriori u.1} we infer that
$ \lim_{T\to 0+} N_{T}^s(u) \lesssim (1+\|u_0\|_{L^2}) \|u_0\|_{H^s}$ . On the other hand, combining \eqref{apriori u.1}-\eqref{apriori u.2}
and \eqref{apriori w.1}  and the conservation of the $L^2 $-norm we infer that
$$
N_T ^0(u)\lesssim (1+\|u_0\|_{L^2}) \|u_0\|_{L^2}  +  (N_T^0(u))^2+ (N_T^0(u))^3 \, .
$$
By continuity, this  ensures that there exists $ \varepsilon_0>0 $ and $ C_0>0 $ such that $ N_1^0(u) \le  C_0 \varepsilon $ provided
$ \|u_0\|_{L^2}\le\varepsilon\le \varepsilon_0$. Finally, using again \eqref{apriori u.1}-\eqref{apriori u.2} and  \eqref{apriori w.1}, this leads to
$ N^s_1(u) \lesssim \| u_0\|_{H^s} $ provided  $ \|u_0\|_{L^2}\le\varepsilon\le \varepsilon_0$.\\

\subsection{ Lipschitz bound for initial data having the same low frequency part} \label{Uni}
To  prove the uniqueness as well as the continuity of the solution we will derive a Lipschitz bound on the solution map on some affine subspaces of  $ H^s(\R) $ with values in
  $L^\infty_T  H^s(\R)  $. We know from \cite{KT} that such Lipschitz bound does not exist in general in $ H^s(\R) $. Here we will restrict ourself to solutions emanating from initial data having the same low frequency part. This is clearly sufficient to get uniqueness and it will turn out to  be sufficient to get the continuity of the solution as well as the continuity of the flow-map.\\
Let $\varphi_1$, $\varphi_2  \in B_{\epsilon}\cap H^s(\R) $, $s\ge 0 $,  such that $ P_{LO} \varphi_1= P_{LO} \varphi_2 $ and let $u_1$, $u_2$ be two
solutions to \eqref{BO} emanating  respectively from  $\varphi_1$, and $\varphi_2$ that
satisfy \eqref{theo1.1}  on the time interval $[0,T]$, $0<T<1$.  We also assume that the  primitives $ F_1:=F[u_1] $
 and $ F_2:=F[u_2] $ of respectively $ u_1 $ and $ u_2 $ are  such that the associated  gauge functions $W_1$, $w_1$,
respectively $W_2$, $w_2$,  constructed  in
Subsection \ref{GT}, satisfy \eqref{theo1.1b}. Finally, we  assume that
\begin{equation}\label{fd}
  N^0_T(u_i) \le C_0 \varepsilon\le C_0 \varepsilon_0 .
  \end{equation}
 First, by construction, we observe that since $ F(x)-F(y)=\int_x^y u(z) \, dz $, it holds $P_{LO}\int_y^xudz=P_{LO}\Bigl( F(x)-F(y) \Bigr)=
 P_{LO}F(x)-F(y)$. On the other hand, since $ P_{LO} $ and $\partial_x $ do commute, we have $\partial_xP_{LO}F=P_{LO}u$ and, by integrating, $\int_y^xP_{LO}udz=P_{LO}F(x)-P_{LO}F(y)$.  Gathering these two identities, we get
 $$
 \int_y^xP_{LO}udz-P_{LO}\int_y^xudz=F(y)-P_{LO}F(y)=P_{HI}F(y),
 $$
 which leads to
 $$
 P_{lo}\int_y^xudz=P_{lo}\int_y^xP_{LO}udz.
 $$
 We thus infer that
 \begin{eqnarray}
 P_{lo}(F_1-F_2)(x,0) &
= & \int_{\mathbb R}\psi(y)P_{lo}\int_y^x(u_1-u_2)(z,0)dzdy \nonumber \\
&=& \int_{\mathbb R}\psi(y)P_{lo}\int_y^xP_{LO}(\varphi_1(z)-\varphi_2(z))(z,0)dzdy=0.  \label{Plo}
\end{eqnarray}
Then, we set $v=u_1-u_2$, $Z=W_1-W_2$ and $z=w_1-w_2$. Obviously, $z$
satisfies
\begin{displaymath} \label{theo1.7}
\begin{split}
\partial_tz-i\partial_x^2z=&-\partial_xP_{+hi}\big(W_1P_-\partial_xv \big)
-\partial_xP_{+hi}\big(ZP_-\partial_xu_2 \big) \\&
-\partial_xP_{+hi}\big(P_{lo}e^{-\frac{i}{2}F_1}P_-\partial_xv \big)
-\partial_xP_{+hi}\big(P_{lo}\big(e^{-\frac{i}{2}F_1}-e^{-\frac{i}{2}F_2}
\big)P_-\partial_xu_2 \big).
\end{split}
\end{displaymath}
Thus, we deduce gathering estimates \eqref{prop1.2.2},
\eqref{bilincrit1}, \eqref{bilincrit1b}, \eqref{lemma3.1} and
\eqref{lemma3.1b} that
\begin{displaymath}
\begin{split}
\|z\|_{Y^s_1} &\lesssim  \|z(0)\|_{H^s}+
\|w_1\|_{X^{s,1/2}_1}\big(\|v\|_{X_1^{-1,1}}+\|v\|_{L^4_{x,1}}+\|v\|_{L^{\infty}_1L^2_x}\big)+\|v\|_{L^4_{x,1}}^2
\\ & \quad
+\|z\|_{X^{s,1/2}_1}\big(\|u_2\|_{X_1^{-1,1}}+\|u_2\|_{L^4_{x,1}}+\|u_2\|_{L^{\infty}_1L^2_x}\big)
\\ & \quad
+\big(\|v\|_{L^{\infty}_1L^2_x}+\|e^{-\frac{i}{2}F_1}-e^{-\frac{i}{2}F_2}\|_{L^{\infty}_{x,1}}\big)\|u_2\|_{L^4_{x,1}},
\end{split}
\end{displaymath}
which implies recalling \eqref{theo1.2} and \eqref{fd}  that
\begin{equation} \label{theo1.8}
\|z\|_{Y^s_1} \lesssim \|z(0)\|_{H^s}+\varepsilon
\big(\|v\|_{X_1^{-1,1}}+\|v\|_{L^4_{x,1}}+\|v\|_{L^{\infty}_1L^2_x}\big)+\varepsilon
\|e^{-\frac{i}{2}F_1}-e^{-\frac{i}{2}F_2}\|_{L^{\infty}_{x,1}}.
\end{equation}
where, by the mean-value theorem,
\begin{eqnarray*}
\|z(0)\|_{H^{s}} & \lesssim &
\|\varphi_1-\varphi_2\|_{H^{s}}\Bigl(1+\|\varphi_1\|_{H^{s}}
 +\|\varphi_2\|_{L^2}\Bigr)\nonumber\\
 & & +\|e^{-iF_1(0)/2}-e^{-iF_2(0)/2}\|_{L^\infty} \|\varphi_1\|_{H^{s}}
 (1+\|\varphi_1\|_{L^2}) \nonumber \\
  & \lesssim & \|\varphi_1-\varphi_2\|_{H^{s}}+ \| F_1(0)-F_2(0)\|_{L^\infty}.
\label{o5}
\end{eqnarray*}
On the other hand, the equation for $v=u_1-u_2$ reads
\begin{displaymath}
\partial_tv+\mathcal{H}\partial_x^2v=\frac12\partial_x\big((u_1+u_2)v\big),
\end{displaymath}
so that it is deduced from \eqref{apriori u.1bbb}, \eqref{theo1.2}
and the fractional Leibniz rule that
\begin{equation} \label{theo1.9}
\|v\|_{X^{-1,1}_1} \lesssim
\|\partial_tv+\mathcal{H}\partial_x^2v\|_{L^2_1H^{-1}_x}+\|v\|_{L^{\infty}_T L^2_x}
\lesssim \varepsilon\|v\|_{L^4_{x,1}}+\|v\|_{L^{\infty}_1L^2_x}.
\end{equation}

Next, proceeding as in \eqref{gauge3}, we infer that
\begin{displaymath}
\begin{split}
P_{+HI}v &= 2iP_{+HI}\big( e^{\frac{i}{2}F_1}z\big)+ 2iP_{+HI}\big(
(e^{\frac{i}{2}F_1}-e^{\frac{i}{2}F_2})w_2\big) \\ & \quad
+2iP_{+HI}\big(
P_{+hi}e^{\frac{i}{2}F_1}\partial_xP_{+lo}(e^{-\frac{i}{2}F_1}-e^{-\frac{i}{2}F_2})\big)
\\ & \quad +2iP_{+HI}\big(
P_{+hi}(e^{\frac{i}{2}F_1}-e^{\frac{i}{2}F_2})\partial_xP_{+lo}e^{-\frac{i}{2}F_2}\big)
\\ & \quad
+2iP_{+HI}\big(
P_{+HI}e^{\frac{i}{2}F_1}\partial_xP_{-}(e^{-\frac{i}{2}F_1}-e^{-\frac{i}{2}F_2})\big)
\\ & \quad +2iP_{+HI}\big(
P_{+HI}(e^{\frac{i}{2}F_1}-e^{\frac{i}{2}F_2})\partial_xP_{-}e^{-\frac{i}{2}F_2}\big).
\end{split}
\end{displaymath}
Thus, we deduce using estimates \eqref{lemma1.1b}, \eqref{lemma1.6} and arguing as in the
proof of Proposition \ref{apriori u} that
\begin{displaymath} \label{theo1.10}
\begin{split}
\|J^s_xv\|_{L^{p}_1L^q_x} &\lesssim
\big(\|u_1\|_{L^{\infty}_1L^2_x}+\|u_2\|_{L^{\infty}_1L^2_x}\big)\|v\|_{L^{\infty}_1L^2_x}
+(1+\|u_1\|_{L^{\infty}_1L^2_x})\|z\|_{Y^s_1} \\ & \quad
+\big(\|v\|_{L^{\infty}_1L^2_x}+\|e^{\frac{i}{2}F_1}-e^{\frac{i}{2}F_2}\|_{L^{\infty}_{x,1}}
(1+\|u_1\|_{L^{\infty}_1L^2_x}) \big)\|w_2\|_{Y^s_1}\\ & \quad
+\|u_1\|_{L^{\infty}_1L^2_x}\big(\|v\|_{L^{\infty}_1L^2_x}
+\|e^{-\frac{i}{2}F_1}-e^{-\frac{i}{2}F_2}\|_{L^{\infty}_{x,1}}\|u_1\|_{L^{\infty}_1L^2_x}\big)
\\ & \quad
+\|u_2\|_{L^{\infty}_1L^2_x}\big(\|v\|_{L^{\infty}_1L^2_x}
+\|e^{\frac{i}{2}F_1}-e^{\frac{i}{2}F_2}\|_{L^{\infty}_{x,1}}\|u_1\|_{L^{\infty}_1L^2_x}\big),
\end{split}
\end{displaymath}
for $(p,q)=(\infty,2)$ or $(p,q)=(4,4)$, which implies recalling \eqref{fd}  that
\begin{equation} \label{theo1.10}
\|J^s_xv\|_{L^{\infty}_1L^2_x}+\|J^s_xv\|_{L^4_{x,1}} \lesssim \|z\|_{Y^s_1}
+\varepsilon\|e^{-\frac{i}{2}F_1}-e^{-\frac{i}{2}F_2}\|_{L^{\infty}_{x,1}}
+\varepsilon\|e^{\frac{i}{2}F_1}-e^{\frac{i}{2}F_2}\|_{L^{\infty}_{x,1}}.
\end{equation}

Finally, we use the mean value theorem to get the bound
\begin{equation} \label{theo1.11}
\|e^{\pm\frac{i}{2}F_1}-e^{\pm\frac{i}{2}F_2}\|_{L^{\infty}_{x,1}}
\lesssim \|F_1-F_2\|_{L^{\infty}_{x,1}}.
\end{equation}
The following crucial lemma gives an estimate for the right-hand side of \eqref{theo1.11}.
\begin{lemma} \label{lemma4}
It holds that
\begin{equation} \label{lemma4.1a}
\|F_1(0)-F_2(0)\|_{L^{\infty}} \lesssim \|\varphi_1-\varphi_2\|_{L^{2}}.
\end{equation}
and
\begin{equation} \label{lemma4.1}
\|F_1-F_2\|_{L^{\infty}_{x,1}} \lesssim \|v\|_{L^{\infty}_1L^2_x}.
\end{equation}
\end{lemma}
\begin{proof}
\eqref{lemma4.1a} clearly follows from \eqref{Plo} together with Bernstein inequality.
To prove \eqref{lemma4.1} we set  $G=F_1-F_2$, $G_{lo}=P_{lo}G$ and $G_{hi}=P_{hi}G$. Then,
\begin{equation} \label{lemma4.2}
\|G\|_{L^{\infty}_{x,1}} \le \|G_{lo}\|_{L^{\infty}_{x,1}}+\|G_{hi}\|_{L^{\infty}_{x,1}}.
\end{equation}
Observe that, from the Duhamel principle and \eqref{Plo}, $G_{lo}$ satisfies
$$
G_{lo} =   \frac12\int_0^tU(t-\tau)P_{lo}\big((u_1+u_2)v\big)(\tau)d\tau\\
$$
Therefore, it follows using Bernstein and H\"older's inequality that
\begin{equation} \label{lemma4.3}
\|G_{lo}\|_{L^{\infty}_{x,1}} \lesssim \|(u_1+u_2)v\|_{L^{\infty}_1L^1_x}
\lesssim \big(\|u_1\|_{L^{\infty}_1L^2_x}+\|u_2\|_{L^{\infty}_1L^2_x} \big)
\|v\|_{L^{\infty}_1L^2_x}.
\end{equation}
On the other hand, the Bernstein inequality ensures that
\begin{equation} \label{lemma4.4}
\|G_{hi}\|_{L^{\infty}_{x,1}} \lesssim \|\partial_xG_{hi}\|_{L^{\infty}_1L^2_x} \lesssim \|v\|_{L^{\infty}_1L^2_x},
\end{equation}
since $\partial_xG=v$. The proof of Lemma \ref{lemma4} is concluded gathering \eqref{fd}, \eqref{lemma4.2}--\eqref{lemma4.4}.
\end{proof}

Finally, estimates \eqref{theo1.8}--\eqref{lemma4.1} lead to
\begin{eqnarray*}
\lefteqn{\|z\|_{Y^s_1}+\|v\|_{X_1^{s-1,1}}+\|v\|_{L^{\infty}_1H^s_x}+\|J^s_xv\|_{L^4_{x,1}}} \\&
\quad \quad \quad \quad \quad  \lesssim \|\varphi_1-\varphi_2\|_{H^s} + \varepsilon \big( \|z\|_{Y^s_1}+\|v\|_{X_1^{s-1,1}}+\|v\|_{L^{\infty}_1H^s_x}+\|J^s_xv\|_{L^4_{x,1}}\big),
\end{eqnarray*}
Therefore we conclude that there exists $0<\varepsilon_1\le \varepsilon_0 $ such that
\begin{equation} \label{toto}
\|z\|_{Y^s_1}+\|v\|_{X_1^{s-1,1}}+\|v\|_{L^{\infty}_1H^s_x}+\|J^s_xv\|_{L^4_{x,1}} \lesssim \|\varphi_1-\varphi_2\|_{H^s}
\end{equation}
provided  $ u_1 $ and $ u_2 $ satisfy  \eqref{fd} with $ 0<\varepsilon\le \varepsilon_1 $.

\subsection{Well-posedness} \label{Pers}
Let $ u_0\in B_{\varepsilon_1} \cap  H^s(\R) $  and
 consider the sequence of initial data $\{u_0^j\} \subset H^{\infty}(\mathbb R)$,
defined by
\begin{equation} \label{theo1.12}
u_0^j=\mathcal{F}_x^{-1}\big(\chi_{|_{[-j,j]}}\mathcal{F}_xu_0\big), \quad \forall \, j \ge 20.
\end{equation}
Clearly,  $\{u_0^j\} $ converges to $ u_0 $ in $ H^s(\R) $. By the
classical well-posedness theory, the associated sequence of
solutions $\{u^j\}$ is a subset of $C([0,1];H^{\infty}(\mathbb R))$
and according to Subsection \ref{Exist}, it satisfies $ N_1^s(u^j)
\le C_0  \varepsilon_1 $. Moreover, since
 $ P_{LO}  u_0^j=P_{LO} u_0 $ for all $ j\ge 20 $,  it follows from the preceding subsection that
 \begin{equation} \label{theo1.13}
\|u^j-u^{j'}\|_{L^{\infty}_1H^s_x} + \|u^j-u^{j'}\|_{L^4 _1 W^{s,4}_x}+\|w^j-w^{j'}\|_{X^{0,1/2}}
\lesssim \|u_0^j-u_0^{j'}\|_{H^s_x} \, .
\end{equation}
Therefore the sequence $\{u^j\}$ converges strongly in $L^\infty_1 H^s(\R) \cap L^{4}_1 W^{s,4} $ to some function  $ u\in C([0,1]; H^s(\R) $ and $ \{w_j\}_{j\ge 4}  $
 converges strongly to some function $ w $ in $X^{s,1/2} $.  Thanks to these strong convergences it is easy to check that $ u $ is a solution to \eqref{BO}
   emanating from $ u_0 $ and that $ w=P_{hi}(\partial_x(e^{-i\partial_x^{-1} u /2})) $. Moreover from the conservation of the $ L^2(\R) $-norm,
 $ u\in C_b (\R;L^2(\R)) \cap C(\R;H^s(\R)) $.

Now let $ {\tilde u} $ be another solution of \eqref{BO} on $[0,T] $ emanating from $ u_0 $ belonging to the same class  of regularity as $ u $. By using again the scaling argument we can always assume that $
 \|  {\tilde u}\|_{L^\infty_T L^2_x} +\|  {\tilde u}\|_{L^4_{x,T}} \le C_0 \varepsilon_1$.
  Moreover, setting $ {\tilde w}:=P_{+hi} (e^{-i F[ {\tilde u}]}) $,
by the Lebesgue monotone convergence theorem, there exists $ N>0 $ such that
$ \|P_{\ge N} {\tilde w}\|_{X^{0,\frac12}_{T}} \le C_0 \varepsilon_1/2 $. On the other hand, using Lemma \ref{prop1.1}-\ref{prop1.2}, it is easy to check that
 \begin{eqnarray*}
\| (1-P_{\ge N}) {\tilde w}\|_{X^{0,\frac12}_{T}}  &\lesssim &    \|u_0\|_{L^2}+  N T^{\frac14} \|{\tilde u}\|_{L^4_{x,T}}
 \|{\tilde w}\|_{L^4_{x,T}} +\|{\tilde u}\|_{L^4_{x,T}}^2 \\
& \lesssim  &  \|u_0\|_{L^2}+ NT^{\frac14}  \|{\tilde w}\|_{X^{0,\frac12}_T} \|{\tilde u}\|_{L^4_{x,T}}
+ \|{\tilde u}\|_{L^4_{x,T}}^2\quad .
\end{eqnarray*}
Therefore, for $ T>0 $ small enough we can require that $ {\tilde u} $ satisfies the smallness condition \eqref{fd}
 with $ \varepsilon_1 $  and
 thus by  \eqref{toto}, $ {\tilde u}\equiv u $ on $ [0,T] $. This proves the uniqueness result for  initial data belonging to
   $ B_{\varepsilon_1}$.

Next, we turn to the continuity of the flow map. Fix $u_0 \in B_{\epsilon_1}$ and $\lambda>0$ and consider the emanating solution $u \in C([0,1]; H^s(\mathbb R))$. We will prove that if $v_0 \in B_{\epsilon_1}$ satisfies $\|u_0-v_0\|_{H^s} \le \delta$, where $\delta$ will be fixed later, then the solution $v$ emanating from $v_0$ satisfies
\begin{equation} \label{theo1.14}
\|u-v\|_{L^{\infty}_1H^s_x} \le \lambda.
\end{equation}
For $j \ge 1$, let $u_0^j$ and $v_0^j$ be constructed as in \eqref{theo1.12}, and denote by $u^j$ and $v^j$ the solutions emanating from $u_0^j$ and $v_0^j$. Then, it follows by the triangular inequality that
\begin{equation} \label{theo1.15}
\|u-v\|_{L^{\infty}_1H^s_x} \le \|u-u^j\|_{L^{\infty}_1H^s_x}
+\|u^j-v^j\|_{L^{\infty}_1H^s_x}+\|v-v^j\|_{L^{\infty}_1H^s_x}.
\end{equation}
First, according to \eqref{theo1.13}, we can choose $ j_0 $ large enough so that
$$
\|u-u^{j_0}\|_{L^{\infty}_1H^s_x}
+\|v-v^{j_0}\|_{L^{\infty}_1H^s_x}\le 2 \lambda/3 \; .
$$
Second, from the definition of $u^j_0$ and $v^j_0$ in \eqref{theo1.12} we infer that
\begin{displaymath}
\|u_0^j-v_0^j\|_{H^3} \le  j^{3-s}  \|u_0-v_0\|_{H^s} \le j^{3-s}  \delta.
\end{displaymath}
Therefore, by using the continuity of the flow map for smooth initial data, we can choose $\delta>0$,  such that
$$
\|u^{j_0}-v^{j_0}\|_{L^{\infty}_1H^s_x} \le \frac{\lambda}{3},
$$
This concludes the proof of Theorem \ref{theo1}.

\section{Improvement of the uniqueness result for $ s>0$}\label{section5}
In this section we prove that uniqueness holds for initial data $
u_0\in H^s(\R) $, $s>0$, in the class  $ u\in L^\infty_T H^s_x \cap
L^4_{T}W^{s,4}_x $. The  great interest of this result is that we
do not assume any condition on the gauge transform of $ u$ anymore.
Moreover, when $s>\frac14$, the Sobolev embedding $L^{\infty}_TH^s_x
\hookrightarrow L^4_TW^{0+,4}_x$ ensures that uniqueness holds in
$L^\infty_T H^s_x$, and thus the Benjamin-Ono equation is
unconditionally well-posed in $ H^s(\R) $ for $ s>\frac14$.

According to the uniqueness result $ i) $ of Theorem \ref{theo1}, it
suffices to prove that for any solution $ u $  to \eqref{BO} that
belongs to  $ L^\infty_T H^s_x \cap L^4_{T} W^{s,4} $, the
associated gauge function $ w=\partial_x P_{hi}(e^{-\frac{i}2F[u]})
$   belongs to $ X^{0,\frac12}_T $. The proof is based on the
following bilinear estimate that is shown in the appendix :
\begin{proposition} \label{bilin}
Let $s>0$. Then, there exists $0<\delta<\frac{s}{10}$ and
$\theta \in (\frac12,1)$, let us say $\theta=\frac12+\delta$, such
that
\begin{equation} \label{bilin1}
\begin{split}
\|P_{+hi}(WP_-\partial_xu)&\|_{X^{\frac12,-\frac12+2\delta}} \\ &
\lesssim \|W\|_{X^{\frac12,\frac12+\delta}}
\big(\|J^s u\|_{L^2_{x,t}}+\|J^{s}u\|_{L^4_{x,t}}+\|u\|_{X^{s-\theta,\theta}}\big).
\end{split}
\end{equation}
\end{proposition}
First note that by the same scaling argument as in Section \ref{Pers}, for any given
$ \varepsilon>0 $, we can always assume that $\|J^s u\|_{L^\infty_T L^2_x}
 + \| J^s u \|_{L^4_{Tx}} \le  \varepsilon  $  and by \eqref{apriori u.1} it follows that
$\|u\|_{X^{s-\theta,\theta}_T }\lesssim \varepsilon $ for $ 0\le \theta\le 1 $.

Now, since $ u\in L^\infty([0,T];H^s(\R)) \cap L^4_T W^{s,4}_{x} $
and satisfies \eqref{BO}, it follows that $ u_t\in
L^\infty([0,T];H^{s-2}(\R)) $. Therefore  $F:= F[u] \in
L^\infty([0,T];H^{s+1}_{loc}) $ and $
\partial_t F\in  L^\infty([0,T];H^{s-1}_{loc}) $. It ensures that
\begin{equation} \label{Uni1}
W:= P_{hi} (e^{-\frac{i}2F})\in L^\infty([0,T];H^{s+1}(\R))\cap L^4_T W^{s+1,4}_x 
\hookrightarrow X^{1,0} ,
\end{equation}
$e^{-\frac{i}2F} \in
L^\infty([0,T];H^{s+1}_{loc}) $ and   the following calculations are
thus justified:
\begin{eqnarray*}
\partial_t W=\partial_t P_+(e^{-\frac{i}2F})& = & -\frac{i}{2}P_{hi}(F_t e^{-\frac{i}2F}) \\
& = & -\frac{i}{2} P_{hi}\Bigl(e^{-\frac{i}2F} ( -{\mathcal H}
F_{xx}+\frac12F_x^2 ) \Bigr)
\end{eqnarray*}
and
$$
\partial_{xx} W=\partial_{xx}P_{hi}(e^{-\frac{i}2F})=P_{hi}\Bigl(e^{-\frac{i}2F}
(-\frac14F_x^2-\frac{i}{2}F_{xx} ) \Bigr) \quad .
$$
It follows that $ W $ satisfies at least in a distributional sense,
\begin{equation} \label{eqW2}
\left\{\begin{array}{ll} \partial_tW-i\partial_x^2W
=-P_{+hi}\big(WP_-\partial_xu\big)-P_{+hi}\big(P_{lo}e^{-\frac{i}{2}F}
P_{-}\partial_xu\big) \\
W(\cdot,0)=P_{+hi}(e^{-\frac{i}{2}F[u_0]})\quad .
\end{array} \right.
\end{equation}
From \eqref{Uni1} and  Lemma \ref{lemma2} we thus deduce that  $ W \in X^{s,1}_T $, so that,
by interpolation with \eqref{Uni1}, $W\in X^{1/2, 1/2+}_T$.
   But, $ u $ being given in $L^{\infty}_TH^s_x \cap
L^4_TW^{s,4}_x \cap X^{s-\theta,\theta}_T$, on one hand gathering \eqref{prop1.2.1},  the  bilinear estimate  \eqref{bilin1} and \eqref{zq}, we infer that there exists only one solution to \eqref{eqW2} in    $ X^{1/2,\frac{1}{2}+}_T $. Hence, $ w=\partial_x W $ belongs to $ X^{-1/2,\frac{1}{2}+}_T$
and is the unique solution to \eqref{gauge2} in
$X^{-1/2,\frac{1}{2}+}_T $ emanating from  the initial data
$w_0=\partial_x P_{hi} (e^{-\frac{i}{2}F[u_0]}) \in L^2(\R)$. On the
other hand,  according to
Proposition \ref{apriori w}, one can construct a solution to \eqref{gauge2}
emanating from $ w_0$ and belonging to
$ Y^s_T $, by using a Picard iterative scheme. Moreover, using \eqref{BO} and Lemma \ref{lemma2} we can easily check
that this solution belongs to $ X^{-1,1}_T $ and thus by interpolation to
$X^{s-,\frac{1}{2}+}_T \hookrightarrow X^{-1/2, \frac{1}{2}+}_T $.
This ensures that $ w =\partial_x P_{hi} (e^{-iF/2}) $ belongs to $
Y^s_T \hookrightarrow X^{0,1/2}_T $ which concludes the proof.

\section{Continuity of the flow-map for the weak $ L^2$-topology}

In \cite{CK} it is proven that, for any $ t\ge 0 $,  the flow-map $
u_0\mapsto u(t) $ associated to the Benjamin-Ono equation is
continuous from $ L^2(\R) $ equipped with the weak topology into
itself. In this section, we explain how the uniqueness part of
Theorem \ref{theo1} enables to really simplify the proof of this
result by following the approach developed in \cite{GoMo}.

Let $\{u_{0,n}\}_n \subset L^2(\R) $ be a sequence of initial data that
converges weakly to $ u_0 $ in $ L^2(\R) $ and let $ u $ be the
solution emanating from $ u_0 $ given by Theorem \ref{theo1}. From
the Banach-Steinhaus theorem, we know that $ \{u_{0,n}\}_n $ is bounded
in $ L^2(\R) $ and from Theorem \ref{theo1} we know that $ \{u_{0,n}\}_n
$ gives rise to a sequence $ \{u_n\}_n $ of solutions to \eqref{BO}
bounded in $ C([0,1];L^2(\R))  \cap L^4(]0,1[\times \R)$ with an
associated sequence of gauge functions $  \{w_n\}_n $  bounded in $
X^{0,1/2}_1 $. Therefore there exist $ v\in  L^\infty(]0,1[;L^2(\R))
\cap X^{-1,1}_1 \cap L^4(]0,1[\times \R)$ and $ z\in X^{0,1/2}_1 $
such that, up to the extraction of a subsequence,  $ \{u_n\}_n $
converges to $ v$ weakly in $ L^4(]0,1[\times \R) $ and weakly star
in $ L^\infty(]0,1[\times \R) $  and  $  \{w_n\}_n $ converges to $
z $ weakly in $ X^{0,1/2}_1 $.  We now need some compactness on  $
\{u_n\}_n $ to ensure that $  z $ is the gauge transform of $ v$.
In this direction, we first notice, since $ \{w_n\}_n $
is bounded in $ X^{0,1/2}_1 $ and by using the Kato's smoothing effect
injected in Bourgain's spaces framework, that  $ \{D^{\frac14} _x
w_n \}_n $ is bounded in $ L^4_x L^2_1 $.  Let $
\eta_R(\cdot):=\eta(\cdot/R) $.  Using \eqref{gauge3} and Lemma
\ref{lemma2} we infer that
\begin{align*}
\| D^{\frac14}_x   & P_{+HI}u_n\|_{L^2(]0,1[\times ]-R,R[}
\lesssim \| D_x^{\frac14}P_{+HI}\big(e^{\frac{i}{2}F[u_n]}w_n \eta_{R}\big)\|_{L^2_{1,x}}\\
& +\|D_x^{\frac14} P_{+HI}\big(P_{+hi}e^{\frac{i}{2}F[u_n]}\partial_xP_{lo}e^{-\frac{i}{2}F[u_n]}\big)\|_{L^2_{1,x}} \\
&   +\|D_x^{\frac14} P_{+HI}\big(P_{+HI}e^{\frac{i}{2}F[u_n]}\partial_xP_{-hi}e^{-\frac{i}{2}F[u_n]}\big)\|_{L^2_{1,x}}\\
& \lesssim   \|D^{\frac14}_x (w_n \eta_R) \|_{L^{2}_x L^2_1} + \|
D_x^{\frac14} e^{iF[u_n]} \|_{L^8_{1,x}} \| w_n
\|_{L^{\frac83}_{1,x}} + \| u_n\|_{L^4_{1,x}}^2
\end{align*}
But clearly,  $$\|D^{\frac14}_x (w_n \eta_R) \|_{L^{2}_x
L^2_1}\lesssim C(R) (\| D^{\frac14} _x w_n\|_{L^4_x L^2_1} +\|
w_n\|_{L^2_{1,x}}) $$ and by interpolation  $ \| D_x^{\frac14}
e^{iF[u_n]} \|_{L^8_{1,x}}\lesssim \|u_n\|_{L^2_{1,x}}^{\frac34} $.
Therefore, recalling that the $ u_n $ are real-valued functions, it
follows  that $ \{  u_n\}_n $ is bounded in $L^2_1
H^{\frac14}(]-R,R[) $.

Since, according to the equation \eqref{BO},
$ \{\partial_t u_n\}_n $ is bounded in
$ L^2_1 H^{-2}_x $, Aubin-Lions compactness theorem and  standard diagonal extraction arguments ensure that there exists an increasing sequence of integer $ \{n_k\}_k $ such that $  u_{n_k} \to v  $ a.e. in $  ]0,1[\times \R $ and  $ u^2_{n_k} \rightharpoonup  v^2 $ in $ L^2(]0,1[\times \R) $.
In view of  our construction of the primitive $ F[u_n] $ of $ u_n$ (see Section \ref{GT}),
it is then easy to check  that $ F[u_{n_k}] $ converges  to the primitive $F[v]$ of $ v $  a.e. in $ ]0,1[\times \R $. This ensures that $   P_{+hi} (e^{-\frac{i}{2}  F[u_{n_k}]}) $ converges weakly to $
  P_{+hi} (e^{-\frac{i}{2} F[v]})  $ in $ L^2(]0,1[\times \R) $ and thus $ z $ is the gauge transform of $ v$ .
Passing to the limit in the equation ,  we conclude that $ v$
satisfies  \eqref{BO} and belong the class of uniqueness of Theorem
\ref{theo1}.

Moreover, setting  $(\cdot,\cdot)$ for the $L^2_x$ scalar
product, by (\ref{BO}) and the bounds above, it is easy to check
that, for any smooth space function $\phi $ with compact support,
the family $
 \{t\mapsto (u_{n_k}(t), \phi)\} $ is uniformly
equi-continuous on $ [0,1] $. Ascoli's theorem then ensures that
$ (u_{n_k}(\cdot),\phi)$ converges to $ (v(\cdot),\phi) $
uniformly  on $ [0,1] $ and thus $v(0)=u_0$. By uniqueness, it
follows that $ v\equiv u $  which ensures that the whole sequence $ \{u_n\} $ converges to $ v $ in  the sense above and not only a subsequence.  Finally, from the above convergence result,
it results that $u_n(t) \rightharpoonup u(t) $ in $L^2_x$ for all
$ t\in [0,1] $.\qed

\section{The periodic case}
In this section we explain how the bilinear estimate proved in Proposition
 \ref{bilincrit} can lead to a great simplification of the global well-posedness result in $ L^2(\T) $ derived  in \cite{Mo2} and to  new uniqueness results in $ H^s(\T) $, where $ \T=\R/2\pi \Z $. With the notations of \cite{Mo} these new results lead to the following  global well-posedness theorem  :
\begin{theorem} \label{theo1per} Let $s \ge 0$ be given. \mbox{ } \\
\underline{\it Existence :}
 For all $u_0 \in H^s(\T)$ and all $ T>0 $, there exists a
 solution
\begin{equation} \label{theo1.1}
u \in C([0,T];H^s(\T)) \cap X^{s-1,1}_T \cap L^4_TW^{s,4}(\T)
\end{equation}
of \eqref{BO} such that
\begin{equation} \label{theo1.1b}
w=\partial_xP_{+hi}\big(e^{-\frac{i}{2}\partial_x^{-1} {\tilde u}} \big) \in Y_T^s.
\end{equation}
 where
$$
 {\tilde u}:=u(t,x-t \int  \hspace*{-4mm}- u_0)-\int \hspace*{-4mm}- u_0 \quad \mbox{ and  }\quad
 \widehat{\partial_x^{-1}}:= \frac{1}{i\xi}, \, \xi\in \Z^* \quad .
 $$\vspace{2mm}\\
\underline{\it Uniqueness :} This solution is unique in the following classes :
$$
\begin{array}{lll}
i) & u\in L^\infty(]0,T[;L^2(\T) \cap L^4(]0,T[\times\T) \mbox{ and } w\in X^{0,\frac12}_T .\\
ii) &  u\in L^\infty(]0,T[;H^{\frac14}(\T) \cap L^4_T W^{\frac14,4}(\T) & \mbox{ whenever } s\ge \frac14.\\
iii)  &  u\in L^\infty(]0,T[;H^{\frac12}(\T)  & \mbox{ whenever } s\ge \frac12 .
\end{array}
$$
Moreover, $u\in C_b(\mathbb R;L^2(\T))$ and the flow map
data-solution $:u_0 \mapsto u$ is continuous from $H^s(\T)$
into $C([0,T];H^s(\T))$.
\end{theorem}
\noindent {\bf Sketch of the proof.} In the periodic case, following \cite{Mo}, the gauge transform is defined as follows :
Let $ u $ be a smooth $2\pi  $-periodic  solution of (BO) with initial data $ u_0 $.
 In the sequel,  we will  assume that $u(t)  $ has  mean value zero for all time.
  Otherwise we do  the change of unknown :
 \begin{equation}
 {\tilde u}(t,x):=u(t,x-t \int \hspace{-4mm}- u_0) -\int  \hspace{-4mm}-  u_0 \label{chgtvar}\quad  ,
 \end{equation}
 where $ \int  \hspace{-4mm}- u_0:=\frac{1}{2\pi } \int_{\T } u_0 $ is the mean value of $ u_0 $. It is easy to see that
 $ {\tilde u}$ satisfies (BO) with $ u_0-\int  \hspace{-3mm}- u_0 $ as initial data and since $ \int \hspace{-3mm}-  {\tilde u} $
  is preserved by the flow of (BO), $  {\tilde u}(t) $ has mean value zero for all time.
 We  take for the  primitive of $ u$   the unique periodic, zero mean value, primitive of $ u $ defined by
 $$
 \hat {F}(0) =0 \quad \mbox{ and } \widehat{F}(\xi)=\frac{1}{ i\xi} \hat{u}(\xi) , \quad \xi\in \Z^* \quad .
 $$
The gauge transform is then defined by
\begin{equation}
W:=P_+(e^{-iF/2}) \quad . \label{defW}
\end{equation}
Since $ F $ satisfies $$ F_t +{\mathcal H} F_{xx}=\frac{F_x^2}{2}-\frac{1}{2} \int  \hspace{-4mm}-  F_x^2=
  \frac{F_x^2}{2}-\frac{1}{2} P_0(F_x^2) \quad , $$
we finally obtain that $ w:=W_x=-\frac{i}{2} P_{+hi}(e^{-iF/2} F_x)
=-\frac{i}{2} P_+(e^{-iF/2} u)$ satisfies
\begin{eqnarray}
w_t-iw_{xx} & = & -\partial_x P_{hi}\Bigl[  e^{-iF/2}\Bigl(P_-(F_{xx})-\frac{i}{4} P_0(F_x^2)\Bigr)\Bigr] \nonumber \\
 & =  & -\partial_x P_{+hi} \Bigl(W P_-( u_{x})  \Bigr)+ \frac{i}{4} P_0(F_x^2) w \label{eqw} \; .
\end{eqnarray}
Clearly the second term is harmless and the first one has exactly the same structure as the  one that we estimated in Proposition
 \ref{bilincrit} . Following carefully the proof of this proposition, it is not too hard to check that it also holds in the periodic case independently of the period $ \lambda\ge 1 $. Note in particular that \eqref{bourgstrich.2} also holds with $ L^4_{x,t}  $ and $ X^{0,\frac38} $ respectively replaced by $ L^4_{t,\lambda}$
  and $  X^{0,\frac38}_\lambda $, $\lambda\ge 1 $, where the subscript $\lambda $ denotes  spaces of functions  with space variable  on the torus $ \R/2\pi \lambda \Z $ (see \cite{Bo1} and also \cite{Mo}). This leads to a great simplification of the proof the global well-posedness in $ L^2(\T) $ proved in \cite{Mo2}.

 Now to derive the new uniqueness result we proceed exactly as in Section \ref{section5} except that Proposition \ref{bilin} does not hold on the torus. Actually,  on the torus it should be replaced by   \begin{proposition} \label{bilinperiodic}
   For $ s\ge \frac14 $ and all $ \lambda\ge 1 $  it holds
 \begin{equation} \label{bilintore}
\begin{split}
\|P_{+hi}(WP_-\partial_xu)&\|_{X^{s+\frac12,-\frac12}_{\lambda}} \\ &
\lesssim \|W\|_{X^{s+\frac12,\frac12}_{\lambda}}
\big(\|J_x^{s} u\|_{L^2_{T,\lambda}}+\|J_x^{s}u\|_{L^4_{T,\lambda}}+\|u\|_{X^{s-1,1}_\lambda}\big).
\end{split}
\end{equation}
\end{proposition}
Going back to  the proof of the bilinear estimate it easy to be convinced that the above estimate works at the level
 $ s=0+ $ in the regions $ {\mathcal A}$ and ${\mathcal B}$ (see the proof of Proposition \ref{bilin}), whereas in the region ${\mathcal C}$ we are clearly in trouble.
 Indeed, when $s=0$, \eqref{ed}  has then to be replaced by
 $$
 |k^{\frac12} k_1^{-\frac12}k_2^2\langle\sigma_2\rangle^{-1}| \sim |k^{-\frac12}k_1^{-\frac12} k_2 |
 $$
 which cannot be bound when $|k_2|>>k $. On the other hand  at the level $ s=\frac14 $  it becomes
   $$  |k^{\frac34} k_1^{-\frac34}k_2^{\frac74}\langle\sigma_2\rangle^{-1}| \sim |k^{-\frac14}k_1^{-\frac34} k_2^{\frac34} | \lesssim
    k^{- \frac14} \lesssim 1 \,
   $$
   which yields the result.

 With Proposition \ref{bilinperiodic} in hand, exactly the same procedure as in Section \ref{section5} leads to
  the uniqueness result in the class   $ u\in L^\infty_T H^{\frac14} (\T) \cap L^4_T W^{\frac14,4}(\T) $ and by Sobolev embedding to the uniqueness in the class $ u\in L^\infty_T H^{\frac12} (\T)$, i.e. unconditional uniqueness in $ H^{\frac12}(\T)$. As in the real line case, it proves the uniqueness of the (energy) weak solutions that belong to $L^\infty(\R; H^{1/2}(\T))$.

\section*{Appendix}

\subsection*{Proof of Proposition \ref{bilin}}

We will need that  following calculus lemma stated in \cite{GTV}.
\begin{lemma} \label{calculus}
Let $0< a_- \le a_+$ such that $a_-+a_+>\frac12$. Then, for all $\mu \in \mathbb R$
\begin{equation} \label{calculus1}
\int_{\mathbb R} \langle y\rangle^{-2a_-}\langle y-\mu\rangle^{-2a_+}dy \lesssim \langle \mu\rangle^{-s},
\end{equation}
where $s=2a_-$ if $a_+>\frac12$, $s=2a_--\epsilon$, if $a_+=\frac12$, and $s=2(a_++a_-)-1$, if $a_+<\frac12$ and $\epsilon$ denote any small positive number.
\end{lemma}
 The proof  of Proposition \ref{bilin} follows closely the one of Proposition \ref{bilincrit} except in the region
  $ \sigma_2 $-dominant where we use the approach developed in \cite{KPV}.
Recalling the notation used in \eqref{bilincrit3bb}--\eqref{bilincrit3bbb},
we need to prove that
\begin{equation} \label{bilin2}
\big|K\big| \lesssim \|h\|_{L^2_{x,t}}\|f\|_{L^2_{x,t}}\big(\|u\|_{L^2_{x,t}}+\|u\|_{L^4_{x,t}}
+\|u\|_{X^{-\theta,\theta}}\big),
\end{equation}
where
\begin{equation} \label{bilin3}
K=\int_{\mathcal{D}}\frac{\langle \xi \rangle^{\frac12}}
{\langle \sigma\rangle^{\frac12-2\delta}}\widehat{h} (\xi,\tau)
\frac{\langle \xi_1 \rangle^{-\frac12}}{\langle \sigma_1
\rangle^{\frac12+\delta}}\widehat{f}(\xi_1,\tau_1)
\xi_2\langle \xi_2 \rangle^{-s} \widehat{u}(\xi_2,\tau_2)d\nu.
\end{equation}

For the same reason as in the proof of Proposition \ref{bilincrit}, we can assume that
$|\xi_2| \le 1$. By using a Littlewood-Paley decomposition on $h$, $f$ and $u$, $K$ can be rewritten as
\begin{equation} \label{bilin4}
K=\sum_{N,N_1,N_2} K_{N,N_1,N_2}
\end{equation}
with
\begin{displaymath}
K_{N, N_1,N_2}:= \int_{\mathcal{D}}\frac{\langle \xi \rangle^{\frac12}}
{\langle \sigma\rangle^{\frac12-2\delta}}\widehat{P_Nh} (\xi,\tau)
\frac{\langle \xi_1 \rangle^{-\frac12}}{\langle \sigma_1
\rangle^{\frac12+\delta}}\widehat{P_{N_1}f}(\xi_1,\tau_1)
\xi_2\langle \xi_2 \rangle^{-s} \widehat{P_{N_2}u}(\xi_2,\tau_2)d\nu,
\end{displaymath}
and the dyadic numbers $N, \ N_1$ and $N_2$ ranging from $1$ to $+\infty$.
Moreover, we will denote by $K^{\mathcal{A}_{N,N_2}}_{N, N_1,N_2}$, $K^{\mathcal{B}_{N,N_2}}_{N, N_1,N_2}$, $K^{\mathcal{C}_{N,N_2}}_{N, N_1,N_2}$ the restriction of $K_{N, N_1,N_2}$ to the regions $\mathcal{A}_{N,N_2}$, $\mathcal{B}_{N,N_2}$ and $\mathcal{C}_{N,N_2}$ defined in \eqref{bilincrit4}. Then, it follows that
\begin{equation} \label{bilin5}
\big|K\big| \le \big|K_{\mathcal{A}}\big|+\big|K_{\mathcal{B}}\big|+\big|K_{\mathcal{C}}\big|,
\end{equation}
where
$$
K_{\mathcal{A}}:=\sum_{N, N_1,N_2} J^{\mathcal{A}_{N,N_2}}_{N, N_1,N_2}, \; K_{\mathcal{B}}:=\sum_{N, N_1,N_2} K^{\mathcal{B}_{N,N_2}}_{N, N_1,N_2} \ \mbox{and } \ K_{\mathcal{C}}:=\sum_{N, N_1,N_2} J^{\mathcal{C}_{N,N_2}}_{N, N_1,N_2}\; ,
 $$
so that  it suffices to estimate $\big|K_{\mathcal{A}}\big|$, $\big|K_{\mathcal{B}}\big|$ and $\big|K_{\mathcal{C}}\big|$. Recall that, due to the structure of $\mathcal{D}$, one of the following case must hold:
\begin{enumerate}
\item{} high-low interaction: $N_1 \sim N$ and $N_2 \le N_1$
\item{} high-high interaction: $N_1 \sim N_2$ and $N \le N_1$.
\end{enumerate}

\noindent \textit{Estimate for $\big|K_{\mathcal{A}}\big|$.}
In the first case, it follows from the triangular inequality, Plancherel's identity and
H\"older's inequality that
\begin{displaymath}
\begin{split}
\big|K_{\mathcal{A}}\big| &\lesssim  \|h \|_{L^2_{x,t}}  \sum_{N_1} \sum_{N_2 \le N_1}
\frac{N_1^{\frac12}}{(N_1N_2)^{\frac12-2\delta}}
  \Bigr\|P_{N_1}\big(J_x^{-\frac12}P_{N_1}\Big( \frac{\widehat{f}}{\langle \sigma_1
\rangle^{\frac12+\delta}} \Big)^{\vee}P_-\partial_xJ^{-s}_xP_{N_2}u\big)\Bigl\|_{L^2_{x,t}} \\
& \lesssim
\|h \|_{L^2_{x,t}}  \sum_{N_1} \sum_{N_2 \le N_1}
\frac{N_2^{\frac12-s+2\delta}}{(N_1)^{\frac12-2\delta}}
  \bigr\|P_{N_1}\Big( \frac{\widehat{f}}{\langle \sigma_1
\rangle^{\frac12+\delta}} \Big)^{\vee}\bigl\|_{L^4_{x,t}}
\|P_{N_2}u\|_{L^4_{x,t}} \\
& \lesssim
\|h \|_{L^2_{x,t}}\|u\|_{L^4_{x,t}}\sum_{N_1}N_1^{4\delta-s}
\bigr\|P_{N_1}\Big( \frac{\widehat{f}}{\langle \sigma_1
\rangle^{\frac12+\delta}} \Big)^{\vee}\bigl\|_{L^4_{x,t}}.
\end{split}
\end{displaymath}
Then, it is deduced from the Cauchy-Schwarz inequality in $N_1$ that
\begin{equation} \label{bilin6}
\big|K_{\mathcal{A}}\big| \lesssim \|h \|_{L^2_{x,t}}\Big(\sum_{N_1}
\bigr\|P_{N_1}\Big( \frac{\widehat{f}}{\langle \sigma_1
\rangle^{\frac12+\delta}} \Big)^{\vee}\bigl\|_{L^4_{x,t}}^2\Big)^{\frac12}
\|u\|_{L^4_{x,t}},
\end{equation}
since $s>10 \delta$. On the other, estimate \eqref{bilin6} also holds in the case
of high-high interaction by arguing exactly as in \eqref{bilincrit7}, so that estimate
\eqref{bourgstrich.2} yields
\begin{equation} \label{bilin7}
\big|K_{\mathcal{A}}\big| \lesssim \|h \|_{L^2_{x,t}}
\|f \|_{L^2_{x,t}}\|u\|_{L^4_{x,t}}.
\end{equation}

\noindent \textit{Estimate for $\big|K_{\mathcal{B}}\big|$.}
The estimate
\begin{equation} \label{bilin8}
\big|K_{\mathcal{B}}\big| \lesssim \|h \|_{L^2_{x,t}}
\|f \|_{L^2_{x,t}}\|u\|_{L^4_{x,t}},
\end{equation}
follows arguing as in \eqref{bilin6}.

\noindent \textit{Estimate for $\big|K_{\mathcal{C}}\big|$.}
First observe that
\begin{equation} \label{bilin9}
\big|K_{\mathcal{C}}\big| \lesssim \int_{\widetilde{\mathcal{C}}}\frac{|\xi|^{\frac12}}
{\langle \sigma\rangle^{\frac12-2\delta}}|\widehat{h} (\xi,\tau)|\frac{|\xi_1|^{-\frac12}}
{\langle \sigma_1\rangle^{\frac12+\delta}}|\widehat{f}(\xi_1,\tau_1)|
 \frac{|\xi_2|^{(1+\theta-s)}}{\langle\sigma_2\rangle^{\theta}}
 \frac{\langle\sigma_2\rangle^{\theta}}{|\xi_2|^{\theta}}|\widehat{u}(\xi_2,\tau_2)|d\nu,
\end{equation}
where
\begin{displaymath}
\widetilde{\mathcal{C}}=\big\{(\xi,\xi_1,\tau,\tau_1) \in \mathcal{D} \ |
\ (\xi,\xi_1,\tau,\tau_1) \in \bigcup_{N,N_2}\mathcal{C}_{N,N_2} \big\}.
\end{displaymath}
Since $|\sigma_2|>|\sigma|$ and $|\sigma_2|>|\sigma_1|$ in $\widetilde{\mathcal{C}}$,
\eqref{bilincrit4} implies that $|\sigma_2| \gtrsim |\xi\xi_2|$. Applying twice the Cauchy-Schwarz inequality, it is deduced that
\begin{displaymath}
\big|K_{\mathcal{C}}\big| \lesssim \sup_{\xi_2, \tau_2}\big(L_{\widetilde{\mathcal{C}}}(\xi_2,\tau_2)\big)^{\frac12}
\|f\|_{L^2_{\xi,\tau}}\|g\|_{L^2_{\xi,\tau}}\|h\|_{L^2_{\xi,\tau}},
\end{displaymath}
where
\begin{displaymath}
L_{\widetilde{\mathcal{C}}}(\xi_2,\tau_2)=\frac{|\xi_2|^{2+2(\theta-s)}}{\langle \sigma_2\rangle^{2\theta}}\int_{\mathcal{C}(\xi_2,\tau_2)}\frac{|\xi|
|\xi_1|^{-1}}
{\langle \sigma\rangle^{1-4\delta}\langle \sigma_1\rangle^{1+2\delta}}d\xi_1d\tau_1,
\end{displaymath}
and
\begin{displaymath}
\widetilde{\mathcal{C}}(\xi_2,\tau_2)=\big\{(\xi_1,\tau_1) \in \mathbb R^2 \ | \ (\xi,\xi_1,\tau,\tau_1) \in \mathcal{C} \big\}.
\end{displaymath}
Thus, to prove that
\begin{equation} \label{bilin11}
\big|K_{\mathcal{C}}\big| \lesssim \|h \|_{L^2_{x,t}}
\|f \|_{L^2_{x,t}}\|u\|_{X^{-\theta,\theta}},
\end{equation}
it is enough to prove that $L_{\widetilde{\mathcal{C}}}(\xi_2,\tau_2) \lesssim 1$ for all $(\xi_2,\tau_2) \in \mathbb R^2$. We deduce from \eqref{calculus1} and \eqref{bilincrit4} that
\begin{displaymath}
L_{\widetilde{\mathcal{C}}}(\xi_2,\tau_2)\lesssim\frac{|\xi_2|^{2+2(\theta-s)}}{\langle \sigma_2\rangle^{1+2\delta}}\int_{\xi_1}\frac{|\xi|
|\xi_1|^{-1}}
{\langle \sigma_2+2\xi\xi_2\rangle^{1-4\delta}}d\xi_1,
\end{displaymath}
since $\theta=1+\delta$. To integrate with respect to $\xi_1$, we change variables
\begin{displaymath}
\mu_2=\sigma_2+2\xi\xi_2 \quad \text{so that} \quad d\mu_2=2\xi_2d\xi_1 \quad \text{and}
\quad |\mu_2| \le 4|\sigma_2|.
\end{displaymath}
Moreover, \eqref{bilincrit3b} and \eqref{bilincrit4} imply that
\begin{displaymath}
\frac{|\xi||\xi_1|^{-1}|\xi_2|^{1+2(\theta-s)}}{|\xi_1|^2} \le |\xi\xi_2|^{\frac12+\theta-s}
\lesssim |\sigma_1|^{\frac12+\theta-s}
\end{displaymath}
in $\widetilde{\mathcal{C}}$. Then,
\begin{displaymath}
\begin{split}
L_{\widetilde{\mathcal{C}}}(\xi_2,\tau_2) &\lesssim \frac{|\xi_2|^{1+2(\theta-s)}}{\langle \sigma_2\rangle^{1+2\delta}}
\int_0^{4|\sigma_2|}\frac{|\xi|
|\xi_1|^{-1}}
{\langle \mu_2 \rangle^{1-4\delta}}d\mu_2 \\
&\lesssim \frac{\langle \sigma_2\rangle^{\frac12+\theta-s+4\delta}}{\langle \sigma_2\rangle^{1+2\delta}}
 \lesssim \langle \sigma_2\rangle^{3\delta-s} \lesssim 1,
\end{split}
\end{displaymath}
since $s-3\delta>0$.

Finally, we conclude the proof of Proposition \ref{bilin} gathering \eqref{bilin2}, \eqref{bilin5}, \eqref{bilin7}, \eqref{bilin8} and \eqref{bilin11}.

\vspace{0,5cm}

\noindent \textbf{Acknowledgments.} This work was initiated during a visit of the second author at the L.M.P.T., Université François Rabelais, Tours. He would like to thank the L.M.P.T for the kind hospitality. L.M. was partially supported by the ANR project "Equa-Disp".

\bibliographystyle{amsplain}

\end{document}